\documentclass[12pt,reqno]{amsart}
\usepackage{hyperref}
\hypersetup{colorlinks=true, linkcolor=blue, citecolor=red, urlcolor=blue}

\usepackage{amssymb,amsmath,amsfonts,latexsym}
\usepackage{ragged2e}
\usepackage{bm}
\usepackage{mathtools}
\usepackage{enumerate} 
\usepackage{array, graphics,xcolor}
\allowdisplaybreaks
\usepackage{mathrsfs}
\usepackage{mathtools}
\usepackage{blkarray, bigstrut}
\usepackage{soul}
\usepackage[mathcal]{eucal}

\usepackage[autostyle=false, style=english]{csquotes}
\usepackage{ulem}

\MakeOuterQuote{"}

\numberwithin{equation}{section}

\allowdisplaybreaks

\newtheorem{thm}{Theorem}[section]
\newtheorem{crl}[thm]{Corollary}
\newtheorem{lma}[thm]{Lemma}
\newtheorem{ppsn}[thm]{Proposition}

\theoremstyle{definition}

\newtheorem{dfn}[thm]{Definition}
\newtheorem{Notation}[thm]{Notation}
\newtheorem{xmpl}[thm]{Example}

\theoremstyle{remark}
\newtheorem{rmrk}[thm]{Remark}
\newtheorem{hypo}[thm]{Hypothesis}

\newtheorem{case}{Case}

\newcommand{\R}{\mathbb{R}}
\newcommand{\N}{\mathbb{N}}
\newcommand{\C}{\mathbb{C}}
\newcommand{\Z}{\mathbb{Z}}

\newcommand{\Rn}{\mathbb{R}^n}
\newcommand{\RRn}{\mathfrak{R}(\Rn)}

\newcommand{\hil}{\mathcal{H}}
\newcommand{\il}{\CMcal{I}}

\newcommand{\ideal}{\CMcal{I}^{1/2}}
\newcommand{\hils}{\mathcal{B}_2(\hil)}
\newcommand{\boh}{\mathcal{B}_1(\hil)}

\newcommand{\bh}{\mathcal{B}(\hil)}

\newcommand{\Tr}{\operatorname{Tr}}

\newcommand{\Tril}{\operatorname{\tau_{\il}}}
\newcommand{\Wtril}{\operatorname{\widetilde{\tau}_{\il}}}

\newcommand{\conv}{\operatorname{conv}}

\newcommand{\clos}{\operatorname{clos}}
\newcommand{\dds}{\dfrac{d}{ds}}
\newcommand{\ddds}{\dfrac{d^{2}}{ds^{2}}}

\newcommand{\cir}{\mathbb{T}}
\newcommand{\la}{\langle}
\newcommand{\ra}{\rangle}
\newcommand{\Hop}{\textbf{H}_{n}}

\newcommand{\V}{\textbf{V}_{n}}

\newcommand{\Wn}{\mathcal{W}_{1}(\R^{n})}
\newcommand{\Wnn}{\mathcal{W}_{2}(\R^{n})}

\newcommand{\Ccrr}{\normalfont{C}_{c}^{n+1}((a,b)^{n})}

\newcommand{\comn}{\normalfont{\text{Com}_{n}}}
\newcommand{\rcomn}{\normalfont{\text{RCom}}}
\newcommand{\W}{\mathcal{W}}

\def\Re{{\mathrm{Re}\,}}
\def\Im{{\mathrm{Im}\,}}

\begin{document}
\title[Trace formulas in higher dimensions]{Trace formulas in higher dimensions}
\author[A. Chattopadhyay]{Arup Chattopadhyay}
\address{Department of Mathematics, Indian Institute of Technology Guwahati, Guwahati, 781039, India}
\email{arupchatt@iitg.ac.in, 2003arupchattopadhyay@gmail.com}

\author[S. Giri]{Saikat Giri}
\address{Department of Mathematics, Indian Institute of Technology Guwahati, Guwahati, 781039, India}
\email{saikat.giri@iitg.ac.in, saikatgiri90@gmail.com}	
	
\author[C. Pradhan]{Chandan Pradhan}
\address{Department of Mathematics, Indian Institute of Science, Bangalore, 560012, India}
\email{chandan.pradhan2108@gmail.com}

\author[A. Usachev]{Alexandr Usachev}
\address{School of Mathematics and Statistics, Central South University, Changsha,  410075, People's Republic of China}
\email{dr.alex.usachev@gmail.com}


\begin{abstract}	
The paper establishes the Krein and Koplienko trace formulas for multivariable operator functions on symmetrically normed ideals of bounded operators. Results are proved for self-adjoint and maximal dissipative operators. They cover both ideals with normal and singular traces. The admissible function classes considered in the trace formulas include both analytic and non-analytic scalar functions. Results are illustrated with examples.
\end{abstract}

\maketitle

\section{Introduction}
Let $\hil$ be a complex separable Hilbert space. Denote by $\bh$, $\boh$ and $\hils$ the algebras of all bounded, trace class and Hilbert-Schmidt class operators on $\hil$ respectively with $\|\cdot\|$, $\|\cdot\|_{1}$, $\|\cdot\|_{2}$ as the associated norms. Let $\il$ be a symmetrically normed ideal (see Definition \ref{Ideal-Def}) of $\bh$ and $\Tril$ be a trace on $\il$ which is $\|\cdot\|_{\il}$-bounded.

The spectral shift measure (or, in short, SSM) is an important concept in the spectral analysis of quantum systems. Looking into its existence has led to significant advancements in mathematical perturbation theory. This measure comes into play when examining how the spectrum of $H_0 + tV$ ($t\in[0,1]$) changes as a bounded self-adjoint perturbation $V$ is gradually applied to a self-adjoint operator $H_0$. The notion of the first order spectral shift measure originated from Lifshitz's work on quantum theory of crystals \cite{Li52}, and the mathematical theory of this concept was later developed by Krein in a series of papers \cite{Kr53, Kr62}. In \cite{Kr53}, Krein proved that given two self-adjoint operators $H_0$ (possibly unbounded) and $V$ such that $V\in\boh$, there exists a unique absolutely continuous (with respect to the Lebesgue measure) measure $\mu$ on $\R$ such that the following trace formula holds:
\begin{align}\label{Kr-Tr}
\Tr\{\phi(H_0+V)-\phi(H_{0})\}=\int_{\R}\phi'(\lambda)d\mu(\lambda),
\end{align}
for every $\phi\in C^{1}(\R)$ whose derivative $\phi'$ admits the representation
\begin{align*}
\phi'(\lambda)=\int_{\R}e^{-\lambda t}d\nu(t)
\end{align*}
for some finite Borel measure $\nu$ on $\R$. In \cite{Pe16}, Peller resolved a longstanding problem posed by Krein. He obtained precise descriptions of the maximal class of functions for which the formula \eqref{Kr-Tr} holds.  In \cite{MN2015}, Malamud and Neidhardt provided a simple formula for the spectral shift function of a pair of self-adjoint extensions via the Weyl function and boundary operators. Note that in \cite{AzDoSu06}, Azamov, Dodds and Sukochev, extended the result of Krein \eqref{Kr-Tr} to a semi-finite von Neumann algebra setting. 
 
Langer extended trace formulas to cover the case of pairs of non-selfadjoint and non-unitary operators \cite{Lang65}. More precisely, he proved the existence of a spectral shift function for pair of operators $H_0+V$, $H_0$ with $V\in\boh$, assuming additionally that their spectra are contained in the open unit disk. Thus, in particular, the result holds for pairs of strict contractions. The assumption on spectra was removed by Krein \cite{Kr87}. He proved the analogue of the trace formula \eqref{Kr-Tr} for pairs of maximal dissipative operators $H_0 + V$, $H_0$ (or that of contractions) with trace class resolvent difference (with a spectral shift measure instead of a spectral shift function). A special case of a self-adjoint $H_0$ was considered by Rybkin \cite{Ryb84}. A complete analogue of the trace formula \eqref{Kr-Tr} for functions of pairs of maximal dissipative operators $H_0+V$, $H_0$ was first established by Malamud and Neidhardt under an additional assumption $\rho(H_0) \cap \rho(H_0+V) \cap \C_{+} \neq \varnothing$ \cite{MN2015}. Later on, Malamud, Neidhardt and Peller managed to remove the assumption of resolvent sets \cite{PellerMD1, PellerMD2}. A complete survey of the known results on this topic can be found in \cite{PellerMD2} (see also \cite{AcSdCp24,MN2015}). Recent publications \cite{ChaSin21, MNP24} described sufficient conditions for a pair of contractions $H_0+V$, $H_0$ with $V\in\boh$ to have a real-valued spectral shift function. In particular, it is the case when $H_0$ is a strong contraction.

The first order SSM has appeared prominently in inverse problems for one-dimensional Schr\"odinger, Dirac, Jacobi, and CMV operators where the respective perturbation is in the form of a boundary condition that changes the domain of $H_0$ (see, e.g., \cite{ClGe02} and references therein). In addition, in \cite{BiKr62}, Birman and Krein established a connection between the first order SSM and scattering theory, while in \cite{AzCaSu07}, Azamov, Carey, and Sukochev discussed significant applications of the first order SSMs in perturbation theory of Schr\"odinger operators and in noncommutative geometry in the study of spectral flow.

The Hilbert-Schmidt perturbation was first considered by Koplienko \cite{Ko84}. Indeed, he proved that for a given pair of self-adjoint operators $H$ and $H_{0}$ such that $V:=H-H_{0}\in\hils$, there exists an absolutely continuous measure $\mu$ on $\R$ satisfying Koplienko's trace formula:
\begin{align}\label{Kop-Tr}
\Tr\Big\{\phi(H)-\phi(H_{0})-\dds\bigg|_{s=0}\phi(H_{0}+sV)\Big\}=\int_{\R}\phi''(\lambda)d\mu(\lambda)
\end{align}
for rational functions $\phi$ with poles off $\R$. The formula \eqref{Kop-Tr} has been significantly extended: an extended version of the formula \eqref{Kop-Tr} now applies to functions $\phi$ in the Besov class $B_{\infty 1}^2(\R)$, as shown by Peller \cite{Pel05}, and further to a broader function class containing $B_{\infty 1}^2(\R)$ as established by Coine et al. in \cite[Theorem 5.1]{CoLemSkSu19} and \cite[Theorem 4.8, Corollary 4.9]{ChCoPrGi24}. For the extension of second order SSMs to the setting of semi-finite von Neumann algebras, see \cite{DySk09}. The Koplienko trace formula for maximal dissipative operators was also studied; for instance, see \cite{AcSdCp24}.

The Krein and Koplienko trace formulas have been extended beyond the trace class and Hilbert-Schmidt class ideals. In 2014, Dykema and Skripka \cite{DySk14} extended these to a specific class of Lorentz ideals $\mathcal{M}_{\psi}$ with a Dixmier trace on it. Dixmier traces are singular traces (that is, traces vanish on finite rank operators) of a special type. Singular traces play a significant role in both classical and noncommutative geometry, as well as in physics (see, e.g., \cite{CaSu06,Connesbook,LoSuZabook}). In contrast to the classical case, these new perturbation formulas introduce spectral shift measures that are not absolutely continuous \cite{DySk14}.  The development of perturbation formulas for singular traces on the weak trace class ideal opens up new possibilities for applications. In particular, \cite{PoSuToZa14} employs singular traces of a Taylor expansion for a specific function $f(t)=t^p$, which serves as a technical tool for studying Fr\'echet differentiability of the $L^p$-norm of Haagerup $L^p$ spaces. More recently, Xia established an analogue of the Krein trace formula for a specific Lorentz ideal $\mathcal{M}_{\psi}$, $\psi(t)=\log(t)$ and Dixmier traces on it \cite{Xia22}. He proved the existence of spectral shift measures which are not necessarily absolutely continuous with respect to the Lebesgue measure. An interested reader can see \cite{PoSuUsZa15} for the recent developments on higher order trace formulas with perturbation operators from the weak Schatten-von Neumann ideal. 

The natural question formulated by Aleksandrov, Peller, and Potapov in \cite[Open problems]{AlPePo19} is whether the classical perturbation formulas have their multivariable counterparts? To be precise, we state their problem (in a broader context):

\medskip

\textit{Let $\Hop:=(H_1,\ldots,H_n)$, $\Hop(1):=(H_1+V_1,\ldots, H_n+V_n)$ be two tuple of commuting self-adjoint operators (not necessarily bounded) in $\hil$ such that $V_i$ are in symmetrically normed ideal $\il$ of $\bh$. Then do there exist finite measures $\mu_{1},\ldots,\mu_{n}$ on $\Rn$ such that
\begin{align}
\label{PellerProblem}&\tau_{\il}\{f(\Hop(1))-f(\Hop)\}=\sum_{j=1}^{n}\int_{\Rn}\frac{\partial f}{\partial \lambda_{j}}(\lambda_{1},\ldots,\lambda_{n})\,d\mu_{j}(\lambda_{1},\ldots,\lambda_{n}),
\end{align}
holds for a nice class of smooth functions $f$ and a trace $\tau_{\il}$ on $\il$?}

\medskip

In the present paper we answer this question in the affirmative under an additional assumption that $\mathbf{V}_{n}:=(V_{1},\ldots,V_{n})$ is commutative. Note that $\Hop,\Hop(1)$ and $\V$ are all commutative if and only if $\Hop(t)=\Hop+t\V$ is commutative for all $t\in[0,1]$, see Lemma \ref{Sufficient-cond-Comn}.

In Theorem \ref{Kr-Thm2} below, we prove formula \eqref{PellerProblem} for perturbations belonging to a symmetrically normed ideal $\il$ of $\bh$ (under Hypothesis \ref{Assumption}(a)) and for the class of Wiener functions. These assumptions are satisfied by many ideals with normal traces (such as the trace class one) as well as those with singular ones. In particular, the result holds for a family of Lorentz ideals (see Theorem \ref{Kr-Thm3} below). In Theorem \ref{Kp-Thm2} we make a step further and prove the Koplienko trace formula for tuples.
 
By means of dilation theory, we establish the Krein and Koplienko trace formulas for pairs of two-tuple of maximal dissipative operators (see Theorems \ref{D-Thm1} and \ref{D-Thm2}). This relies on Ando's result that commuting pairs of contractions admit commuting unitary dilations, which in turn implies that commuting pairs of maximal dissipative operators admit commuting self-adjoint dilations—a property that fails in general for commuting triples. Consequently, we confine ourselves to the two-variable case. Note also that dilation theory was also used in \cite{ChaSin21,PellerMD2}. Our results are applicable to the Hardy space model (see Example \ref{hardy}).

Also, we provide weaker trace formulas for a tuple of operators based on the assumption stated above. These formulas include integral representations involving derivatives of $f$ upto order $n+1$ (for Krein) and $n+2$ (for Koplienko), as detailed in \eqref{P-R5} and \eqref{P-R7}.\smallskip

The cornerstone of the above results is the following estimates for arbitrary ideals $\il$ of $\bh$ (see Theorem \ref{Kr-Thm1} and Theorem \ref{Kp-Thm1}, respectively).

Let $\Hop(1),\Hop$ be as above. Assuming that the path $\Hop(t) = \Hop + t\V, t \in [0, 1]$, joining $\Hop(1)$ and $\Hop$ is commutative, we obtain 
\begin{align}\label{Kr-Est}
&\left|\Tril\left(\frac{\partial f}{\partial \lambda_{j}}(\Hop(t))V_{j}\right)\right|\le(\tau_{1, \CMcal{I}}+\tau_{2,\CMcal{I}}+\tau_{3,\CMcal{I}}+\tau_{4,\CMcal{I}})\left(|V_{j}|\right)\left\|\frac{\partial f}{\partial \lambda_{j}}\right\|_{L^{\infty}(\Rn)},
\end{align}
for all functions $f$ from the Wiener space $\Wn$ (see \eqref{wiener}), whenever $H_{j}\in\bh$, $V_j\in\il$ for $j=1,\ldots,n$ and 
\begin{align}\label{Kop-Est}
\left|\Tril\left(D_{H_{i},H_{j}}^{f}(t)\right)\right|\le\sum_{k=1}^{4} \left(\tau_{k,\il}(|V_{i}|^{2})\right)^{1/2}\,\left(\tau_{k,\il}(|V_{j}|^{2})\right)^{1/2}\,\,\left\|\frac{\partial^{2} f}{\partial \lambda_{i}\,\partial \lambda_{j}}\right\|_{L^{\infty}(\Rn)},
\end{align}
for all functions $f$ from the rational class $\RRn$ (see \eqref{rational}), whenever $V_j\in\ideal$ ($H_{j}$ need not be bounded) for $j=1,\ldots,n$. Here, traces $\tau_{k,\CMcal{I}}~(1\le k\le 4)$ are components of the Jordan decomposition $\Tril=\tau_{1,\il}-\tau_{2,\il} + i\tau_{3,\il}-i\tau_{4,\il}$.

We prove estimate \eqref{Kr-Est} by applying the theory of multiple operator integrals \cite{PellerMOI} and the joint spectral theorem for a commuting tuple of self-adjoint operators. To establish \eqref{Kop-Est}, however, we only require the spectral theorem, as the regularity of the function class in this case allows us to avoid multiple operator integrals.

\medskip

\noindent\textbf{Significance of our results.} There are no work related to the trace formulas for commuting self-adjoint $n$-tuple, except \cite{Sk15} and \cite{ChGiPr24}. In \cite{Sk15}, the author establishes first and second order trace formulas for a pair of $n$-tuple of commuting self-adjoint contractions. Meanwhile, in \cite{ChGiPr24}, the authors proved higher order trace formulas under the strong assumption that the perturbation belongs to the Hilbert-Schmidt class $\hils$. In both cases, the function $f$ belongs to a certain analytic function class. However, it is of interest to see how much we can relax the restrictions on the operators and the function $f$. In our case, we have done both: we consider pair of $n$-tuple of commuting self-adjoint operators rather than focusing solely on contractive self-adjoint operators. For the Krein trace formula, $f$ belongs to the considerably larger class $\mathcal{W}_{2}(\Rn)$ (see \eqref{wiener}), and for the Koplienko trace formula, $f$ belongs to the rational function class $\mathfrak{R}(\R^n)$ (see \eqref{rational}), which are not necessarily analytic. In the special cases of the trace class ideal and Lorentz ideals our results are proved in the following stronger forms.

\begin{enumerate}[{\normalfont(i)}]
\item For $(\il,\Tril)=(\boh,\Tr)$, we can relax the boundedness condition imposed in Theorem  \ref{Kr-Thm2} on the self-adjoint $n$-tuple $\Hop$, see Theorem \ref{Kr-Thm4}. This establishes formula \eqref{PellerProblem} under an extra assumption that $\mathbf{V}_{n}$ is commutative.
	
\item For $(\il,\Tril)=(\mathcal{M}_\psi,\tau_\psi)$, we can relax this additional assumption on $\mathbf{V}_{n}$, see Theorem \ref{Kr-Thm3}. This  establishes formula \eqref{PellerProblem} for a pair of bounded self-adjoint $n$-tuple and for $f\in\Wnn$. Moreover, if $f\in\RRn$, the commutativity assumption in Theorem \ref{Kr-Thm3} can be relaxed even further, see Remark \ref{D-Rmrk1}. This observation provides further motivation for considering singular traces in the context of \eqref{PellerProblem}.
\end{enumerate}

To the best of the authors' knowledge, the multivariable trace formulas for maximal dissipative operators have never appeared in the literature before. We prove the first and the second order formulas in this setting.

The subsequent sections are organized as follows. Section \ref{Sec2} contains necessary definitions and preliminary results on multiple operator integrals, operator ideals, and traces. In Sections \ref{Sec3} and \ref{Sec4}, we establish first and second order trace formulas, respectively, for an $n$-tuple of self-adjoint operators. Section \ref{Sec5} delves into the investigation of first and second order trace formulas for a pair of maximal dissipative operators. The paper concludes with Section \ref{Sec6}, which presents a weaker version of the trace formula.

\section{Preliminaries and preparatory results}\label{Sec2}	

In this section we gather some definitions and known results for the reader's convenience. For more details on symmetrically normed ideals we refer to \cite{GoKr_book,LoSuZabook1ed,SkBook}, on Lorentz ideals and singular traces --- to \cite{Dixmier,LoSuZabook,SuUs16}, on multiple operator integrals --- to \cite{PellerMOI}.

\subsection{Symmetrically normed ideals} We present the following definition from \cite{LoSuZabook1ed}. For a more detailed treatment of the subject, we refer to \cite{GoKr_book}.

\begin{dfn}\label{Ideal-Def}
An ideal $\il$ of $\bh$, equipped with the norm $\|\cdot\|_{\il}$, is called a symmetrically normed ideal if it is a two sided Banach ideal with respect to the ideal norm $\|\cdot\|_{\il}$ and satisfies the following properties:	
\begin{enumerate}[{\normalfont(i)}]
\item $A\in\bh,\,B\in\il,\,0\le A\le B$ implies $A\in\il$ and $\|A\|_{\il}\le\|B\|_{\il}$,
\item there is a constant $k>0$ such that $\|B\|\le k\|B\|_{\il}$ for every $B\in\il$,
\item for all $A, C\in\bh$ and $B\in\il$, we have $\|ABC\|_{\il}\le\|A\|\|B\|_{\il}\|C\|.$
\end{enumerate}
\end{dfn}

The classical examples of symmetrically normed ideals are the Schatten classes:
\[\mathcal{B}_{p}(\hil):=\left\{T\in\bh:\|T\|_{p}=\left(\Tr(|T|^{p})\right)^{1/p}<\infty\right\},\ \ 1\le p<\infty,\]
where $\Tr$ denotes the standard trace on the trace class ideal $\boh$.

\begin{dfn}\label{positivetrace}
A trace on an ideal $\il$ of $\bh$ is a linear functional $\Tril:\il\rightarrow\C$ such that $$\Tril(AB)=\Tril(BA),\,\,\,\,A\in\il,B\in\bh.$$	
The trace is said to be positive if $$A\in\il,\,A\ge 0\,\,\implies\,\,\Tril(A)\ge 0.$$
Suppose $\il$ has the ideal norm $\|\cdot\|_{\il}$, then we say that a trace $\Tril$ is $\|\cdot\|_{\il}$-bounded if there is a constant $M>0$ such that $$|\Tril(A)|\le M\|A\|_{\il},\,\,\,A\in\il.$$
Note that infimum of such constants $M$ equals $\|\Tril\|_{\il^{*}}$.
\end{dfn}

Note that the purely algebraic question of the existence of traces on a given ideal $\il$ is solved in \cite{DyFiWeWo04}. Further results concerning the existence of various types of traces on symmetrically normed ideals can be found in \cite{SuZa13}.

Moving forward we note the following crucial result from \cite[Theorem 4.2.2]{LoSuZabook1ed}: Let $\Tril$ be a $\|\cdot\|_{\il}$-bounded trace on $\il$. Then $\Tril$ can be written in the form
\begin{align}
\label{Split-Pos-Trace}\tau_{\il}=\tau_{1,\il}-\tau_{2,\il}+i\tau_{3,\il}-i\tau_{4,\il},
\end{align}
where all four traces on the right are positive and bounded.

For every $\alpha>0$ and for every ideal $\il$, we consider the root ideals $$\il^{\alpha}=\Big\{A\in\bh:\,|A|^{1/\alpha}\in\il\Big\}.$$
Note that a positive trace $\Tril$ on $\il$ preserves the Cauchy-Schwarz inequality:$$|\Tril(AB)|\le(\Tril(|A|^{2}))^{1/2}(\Tril(|B|^{2}))^{1/2},\,\,\,A,B\in\ideal.$$
For second order trace formulas, we will ask that $\ideal$ be a symmetrically normed ideal equipped with the norm $\|\cdot\|_{\ideal}$ such that
\begin{equation}\label{idealinequality}
\|AB\|_{\il}\le\|A\|_{\ideal}\|B\|_{\ideal},\,\,A,B\in\ideal.
\end{equation}
A sufficient condition for $\ideal$ to be a normed ideal and for inequality \eqref{idealinequality} to hold is established in \cite[Proposition 2.5]{DySk14}.

\subsection{Lorentz ideals and singular traces}\label{Sec-Dixmier} 

For a concave function $\psi:(0,\infty)\to(0,\infty)$ such that $\lim_{t\to\infty} \psi(t)=\infty$ define the Lorentz ideals as follows:
$$\mathcal{M}_\psi=\bigg\{A\in\mathcal{B}(\hil)~\mbox{is compact}:~\|A\|_{\mathcal{M}_\psi}:=\sup_{n\ge0}\frac{1}{\psi(1+n)}\sum_{k=0}^{n}s(k,A)<\infty\bigg\},$$
where $\{s(k,A)\}_{k=0}^{\infty}$ is a singular value sequence of $A$. More information about singular values can be found in \cite[Chapter II]{GoKr_book}. Assuming additionally that 
\begin{equation}
\label{psi-lim-cond}
\lim_{t\to\infty}\frac{\psi(2t)}{\psi(t)}=1,
\end{equation}
J. Dixmier \cite{Dixmier} proved that for a dilation invariant state $\omega$ on the algebra of all bounded sequences, the functional 
\begin{align}\label{Dixmier-Tr}
\Tr_{\omega}(A)=\omega\bigg(\bigg\{\frac{1}{\psi(1+n)}\sum_{k=0}^{n} s(k,A)\bigg\}_{n\ge0}\bigg), \ 0 \le A\in \mathcal{M}_\psi
\end{align}
extends by linearity to the whole ideal $\mathcal{M}_\psi$. It is not difficult to see that this extended functional is a non-trivial trace and that it is not a scalar multiple of the standard trace.  Nowadays, functionals of the form \eqref{Dixmier-Tr} are termed Dixmier traces. They were the first examples of singular traces (see, e.g., \cite[Definition 1.3.2]{LoSuZabook}).

It turned out that the limiting condition on the function $\psi$ assumed above is not optimal \cite{DPSS}. More precisely, the ideal $\mathcal{M}_\psi$ admits non-trivial Dixmier traces if and only if 
\begin{equation*}
\liminf_{t\to\infty}\frac{\psi(2t)}{\psi(t)}=1.
\end{equation*}
For more on Dixmier traces we refer the reader to \cite{LoSuZasurvey,LoSuZabook,SuUs16}. 

Later on, it became clear that there are plenty of singular traces other than Dixmier traces (see, e.g., \cite[Theorem 4.1.4]{LoSuZabook}).

The following result extends \cite[Proposition 2.6]{DySk14} where it was proved for Dixmier traces. Although the proof stays almost the same, we provide it for the reader's convenience.

\begin{ppsn}
Let $\il=\mathcal{M}_\psi$ be a Lorentz ideal and let $\tau_{\psi}$ be a bounded singular trace on it. If for some $0<\varepsilon<1$ there exists a constant $C>0$ such that
\begin{equation}
\label{psi-growth-condition}
\psi(t)\le C t^\varepsilon, \ t \ge 1,
\end{equation}
then 
\begin{equation}
\label{Dixmier-Property}
\tau_\psi(\il^\alpha)=\{0\}\quad\text{for every}\quad\alpha>\frac1{1-\varepsilon}.
\end{equation}
\end{ppsn}

\begin{proof}
Let $A\in\il^\alpha$.
Using the fact that the singular value sequence is decreasing, for every $n\ge0$ we obtain
$$s(n, A)^\frac1\alpha = s(n, A^\frac1\alpha)\le \frac{1}{n+1}\sum_{k=0}^{n}s(k,A^\frac1\alpha)\le \frac{\psi(1+n)}{n+1}\|A^\frac1\alpha\|_\il\le C (1+n)^{\varepsilon-1}\|A^\frac1\alpha\|_\il.$$
Thus, 
$$s(n,A)\le\text{const}\cdot(1+n)^{-\alpha(1-\varepsilon)},$$
where the constant is independent of $\varepsilon$. Using the assumption on $\alpha$ we conclude that the singular value sequence is summable. Hence, the operator $A$ belongs to the trace-class. Since $\tau_\psi$ is bounded singular trace, then $\tau_\psi(A) =0$ \cite[Lemma 2.5.3]{LoSuZabook}.
\end{proof}

\begin{xmpl}
There are many Lorentz ideals satisfying condition \eqref{psi-growth-condition}. Indeed, condition \eqref{psi-lim-cond} implies \eqref{psi-growth-condition} for every $\varepsilon >0$ (see, e.g., \cite[p. 1113]{DySk14}). In particular, these conditions are met by the Lorentz ideal with $\psi(t)=\log(t)$.
\end{xmpl}

The following hypotheses will be in force throughout the remainder of this paper.

\begin{hypo}\label{Assumption}
We consider perturbation operators as an elements of ideal $\il$ or its root ideal $\ideal$, where $\il$ and $\ideal$ satisfy the following conditions:
\begin{enumerate}[{\normalfont(a)}]
\item $\il$ is a symmetrically normed ideal of $\bh$, equipped with the norm $\|\cdot\|_{\il}$, and $\Tril$ is a $\|\cdot\|_{\il}$-bounded trace on $\il$.
\item $\ideal$ is a symmetrically normed ideal with norm $\|\cdot\|_{\ideal}$ and satisfies the inequality \eqref{idealinequality}.
\end{enumerate}
\end{hypo}

\begin{rmrk}
We emphasize here that for our first order trace formulas, we only require $\il$ to satisfy condition (a) of Hypothesis \ref{Assumption}; no condition on $\ideal$ is needed. For the second order trace formulas, we additionally require $\ideal$ to satisfy condition (b).
\end{rmrk}

\begin{xmpl}
Note that the trace class ideal $\il=\boh$, associated with the standard trace $\Tr$, satisfies Definitions \ref{Ideal-Def} and \ref{positivetrace}, as well as the H\"older-like inequality \eqref{idealinequality}. Similarly, the Lorentz ideal $\il=\mathcal{M}_\psi$ also satisfies Definition \ref{Ideal-Def} and the inequality \eqref{idealinequality} (\cite[Proposition 2.5]{DySk14}).
\end{xmpl}

\subsection{Function spaces} Let $m,n\in\N$. Let $C^m(\R^n)$ denote the space of all $m$-times continuously differentiable functions on $\R^n$, $C_{0}(\R^{n})$ the space of continuous functions on $\R^{n}$ that decay to $0$ at infinity, and $C^{m}_{c}(\R^n)$ the subclass of $C^m(\R^n)$ consisting of compactly supported functions. Let $a, b\in\R$. We denote by $C_c^m((a,b)^n)$ the subspace of $C_c^m(\R^n)$ consisting of functions whose closed support is contained in $(a,b)^n$. By $L^{p}(\Rn)$ we denote the standard Lebesgue space on $\Rn$. Given $f\in L^{1}(\Rn)$, we write $\widehat{f}$ for its Fourier transform. Denote the multivariable Wiener space of scalar functions by
\begin{align}\label{wiener}
\mathcal{W}_m(\R^n):=\Big\{ f\in C^m(\R^n):\mathcal{D}^{k}f,\,\widehat{\mathcal{D}^{k}f}\in L^1(\R^n),\,\,\text{for}\,\,0\le k\le m\Big\},
\end{align}              
where $\displaystyle\mathcal{D}^{k}f:=\frac{ \partial^{k}f}{\partial \lambda_1^{i_1}\cdots \partial \lambda_n^{i_n}}$, $i_1,\ldots,i_n\ge 0, i_{1}+\cdots+i_{n}=k$.\\
Let $\mathcal{S}(\Rn)$ denote the class of Schwartz functions on $\Rn$. It follows from \eqref{wiener} that
\[\mathcal{S}(\R^{n})\subset\mathcal{W}_{m}(\Rn)\subset C_{0}(\R^{n}).\]	
Denote the class of rational functions on $\R^{n}$ by
\begin{align}\label{rational}
\mathfrak{R}(\R^n):=\text{span}\left\{ (\lambda_{1},\ldots,\lambda_{n})\in\Rn\mapsto(z-\lambda_{1})^{-k_{1}}\cdots(z-\lambda_{n})^{-k_{n}}:\,\,\Im(z)\neq 0,\,k_{i}\in\N\right\}.
\end{align}
Similarly, denote the subclass by
\begin{align}\label{rationallower}
\mathfrak{R}(\R^{2})^{-}:=\text{span}\left\{(\lambda_{1},\lambda_{2})\in\R^{2}\mapsto(z-\lambda_{1})^{-k_{1}}(z-\lambda_{2})^{-k_{n}}:\,\,\Im(z)< 0,\,k_{1},k_{2}\in\N\right\}.
\end{align}

\subsection{Divided difference}\label{Div-Sec}
We start by revisiting the definition of divided difference of $f:\R\to\C$. For $f\in C^{n}(\R)$, the $n$-th order divided difference $f^{[n]}:\R^{n+1}\to\C$ is defined recursively by
\begin{align*}
f^{[0]}(\lambda)&:=f(\lambda),\\
f^{[n]}(\lambda_{1},\ldots,\lambda_{n+1})&:=\lim_{\lambda\to\lambda_{n+1}}\frac{f^{[n-1]}(\lambda_{1},\ldots,\lambda_{n-1},\lambda)-f^{[n-1]}(\lambda_{1},\ldots,\lambda_{n})}{\lambda-\lambda_{n}}.
\end{align*}	
Now for $f:\Rn\to\C$ we have the following generalization of the divided difference. Consider $f\in C^{k}(\Rn)$. Then the $k$-th order divided difference of $f$ at $j$-th coordinate $f_{j}^{[k]}:\R^{n+k}\to\C$ is defined as follows:
\begin{align}\label{Div-Diff-1}
f_{j}^{[k]}(\lambda_{1},\ldots,\lambda_{j-1},\mu_{1},\ldots,\mu_{k+1},\lambda_{j+1},\ldots,\lambda_{n})=\varphi_{j}^{[k]}(\mu_{1},\ldots,\mu_{k+1}),
\end{align}
where $$\varphi_{j}(\lambda)=f(\lambda_{1},\ldots,\lambda_{j-1},\lambda,\lambda_{j+1},\ldots,\lambda_{n})$$
and $\lambda_{1},\ldots,\lambda_{j-1},\lambda_{j+1},\ldots,\lambda_{n}$ are fixed.

Consider $1\le j<k\le n$. Denote by $f_{j,k}^{[2]}:\R^{n+2}\to\C$ the second order divided difference of $f$ at the $j$-th and $k$-th coordinates, which is defined by
\begin{align}\label{Div-Diff-2}
f_{j,k}^{[2]}(\lambda_{1},\ldots,\lambda_{j-1},\mu_{1},\mu_{2},\lambda_{j+1},\ldots,\lambda_{k-1},\eta_{1},\eta_{2},\lambda_{k+1},\ldots,\lambda_{n})=\lim_{\mu\to\mu_{2}}\frac{\psi_{j}(\mu)-\psi_{j}(\mu_{1})}{\mu-\mu_{1}},
\end{align}
where $$\psi_{j}(\mu)=f_{k}^{[1]}(\lambda_{1},\ldots,\lambda_{j-1},\mu,\lambda_{j+1},\ldots,\lambda_{k-1},\eta_{1},\eta_{2},\lambda_{k+1},\ldots,\lambda_{n}).$$
Here, the first divided difference is taken at the $k$-th coordinate, then at the $j$-th coordinate. Note that \eqref{Div-Diff-2} remains unchanged when the order of the coordinates is interchanged.

\subsection{Function class $\mathfrak{G}(\Rn)$} For a function $f:\Rn\to\C$, we say $f\in\mathfrak{G}(\Rn)$ if there exist a finite measure space $(\Omega,\nu)$ and bounded measurable functions $a_{1},\ldots,a_{n}:\R\times\Omega\to\C$ such that
\begin{align}\label{Function-Rep} f(\lambda_{1},\ldots,\lambda_{n})=\int_{\Omega}a_{1}(\lambda_{1},\sigma)\cdots a_{n}(\lambda_{n},\sigma)\,d\nu(\sigma)
\end{align}
and
\begin{align}\label{Function-Bound}
\int_{\Omega}~\prod_{k=1}^{n}\|a_{k}(\cdot,\sigma)\|_{\infty}\,\,d|\nu|(\sigma)<\infty.
\end{align}

Note that for every $f\in C(\Rn)\cap L^{1}(\Rn)$ with $\widehat{f}\in L^{1}(\Rn)$, we have by the Fourier inversion formula
\begin{align}\label{Fourier-Rep} f(\lambda_{1},\ldots,\lambda_{n})=\int_{\Rn}a_{1}(\lambda_{1},\vec{t})\cdots a_{n}(\lambda_{n},\vec{t})\,d\nu(\vec{t}),
\end{align}
where $a_{k}(\lambda_{k},\vec{t})=e^{it_{k}\lambda_{k}}$ for $1\le k\le n$ and $d\nu(\vec{t})=\frac{\widehat{f}(\vec{t})}{(2\pi)^{n/2}}d\vec{t}$. Here $\vec{t}$ denotes the coordinates $(t_1,\ldots,t_n)$, and $d\vec{t}$ denotes the product measure $dt_1dt_2\cdots dt_n$. Furthermore,
\begin{align*}
\int_{\Rn}~\prod_{k=1}^{n}\|a_{k}(\cdot,\vec{t})\|_{\infty}\,\,d|\nu|(\vec{t})\le\frac{1}{(2\pi)^{n/2}}\|\widehat{f}\|_{L^{1}(\Rn)}<\infty.
\end{align*}
Thus,
$$\Big\{f\in C(\Rn):f,\widehat{f}\in L^{1}(\Rn)\Big\}\subseteq\mathfrak{G}(\Rn).$$

We now present some crucial representations for functions $f\in\Wn$ and $f\in\Wnn$ in the following lemma. The original idea goes back to \cite{AzDo09}.

\begin{lma}\label{Div-Lem}
The following assertions hold:
\begin{enumerate}[{\normalfont(i)}]
\item\label{Div-Lem-R1} Let $f\in\mathcal{W}_{1}(\Rn)$. Then $f_{j}^{[1]}\in\mathfrak{G}(\R^{n+1})$ for every $1\le j\le n$.
\item\label{Div-Lem-R2} Let $f\in\mathcal{W}_{2}(\Rn)$. Then $f_{j}^{[2]}\in\mathfrak{G}(\R^{n+2})$ for every $1\le j\le n$, and $f_{j,k}^{[2]}\in\mathfrak{G}(\R^{n+2})$ for every $1\le j< k\le n$.
\end{enumerate}
\end{lma}

\begin{proof}
Let $\lambda_{1},\ldots,\lambda_{n},\mu_{1},\mu_{2},\eta_{1},\eta_{2}\in\R$. If $f\in\mathcal{W}_{2}(\Rn)$, then \eqref{Div-Diff-2} implies that
\begin{align}
\nonumber&f_{j,k}^{[2]}(\lambda_{1},\ldots,\lambda_{j-1},\mu_{1},\mu_{2},\lambda_{j+1},\ldots,\lambda_{k-1},\eta_{1},\eta_{2},\lambda_{k+1},\ldots,\lambda_{n})\\
\nonumber=&-\int_{\Rn}\int_{0}^{t_{j}}\int_{0}^{t_{k}}\frac{\widehat{f}(\vec{t})}{(2\pi)^{n/2}}\bigg[\prod_{\substack{l=1\\l\neq j,k}}^{n}e^{it_{l}\lambda_{l}}\bigg]\,e^{i(t_{j}-s)\mu_{1}}e^{is\mu_{2}}\,e^{i(t_{k}-p)\eta_{1}}e^{ip\eta_{2}}dp\,ds\,d\vec{t}\\
\label{Div-Lem-R3}=&-\int_{\Pi_{1}}\frac{\widehat{f}(\vec{t})}{(2\pi)^{n/2}}\bigg[\prod_{l=1}^{j-1}e^{it_{l}\lambda_{l}}\bigg]\,e^{i(t_{j}-s)\mu_{1}}e^{is\mu_{2}}\bigg[\prod_{l=j+1}^{k-1}e^{it_{l}\lambda_{l}}\bigg]\,e^{i(t_{k}-p)\eta_{1}}e^{ip\eta_{2}}\bigg[\prod_{l=k+1}^{n}e^{it_{l}\lambda_{l}}\bigg]dp\,ds\,d\vec{t},
\end{align}
where $$\Pi_{1}=\left\{(t_{1},\ldots,t_{n},s,p)\in\R^{n+2}:|s|\le|t_{j}|,|p|\le|t_{k}|,~\text{sign}(s)=\text{sign}(t_{j}),\text{sign}(p)=\text{sign}(t_{k})\right\}.$$
	
Let $\mu_{1},\mu_{2},\mu_{3}\in\R$. Then, \eqref{Div-Diff-1} yields
\begin{align}
\nonumber&f_{j}^{[2]}(\lambda_{1},\ldots,\lambda_{j-1},\mu_{1},\mu_{2},\mu_{3},\lambda_{j+1},\ldots,\lambda_{n})\\
\nonumber=&-\int_{\Rn}\int_{0}^{t_{j}}\int_{0}^{s}\frac{\widehat{f}(\vec{t})}{(2\pi)^{n/2}}\bigg[\prod_{l=1}^{j-1}e^{it_{l}\lambda_{l}}\bigg]\,e^{i(t_{j}-s)\mu_{1}}e^{i(s-p)\mu_{2}}e^{ip\mu_{3}}\bigg[\prod_{l=j+1}^{n}e^{it_{l}\lambda_{l}}\bigg]\,dp\,ds\,d\vec{t}\\
\label{Div-Lem-R4}=&-\int_{\Pi_{2}}\frac{\widehat{f}(\vec{t})}{(2\pi)^{n/2}}\bigg[\prod_{l=1}^{j-1}e^{it_{l}\lambda_{l}}\bigg]\,e^{i(t_{j}-s)\mu_{1}}e^{i(s-p)\mu_{2}}e^{ip\mu_{3}}\bigg[\prod_{l=j+1}^{n}e^{it_{l}\lambda_{l}}\bigg]dp\,ds\,d\vec{t},
\end{align}
where $\Pi_{2}=\left\{(t_{1},\ldots,t_{n},s,p)\in\R^{n+2}:|p|\le|s|\le|t_{j}|,~\text{sign}(p)=\text{sign}(s)=\text{sign}(t_{j})\right\}$.\\
Hence \eqref{Div-Lem-R3} and \eqref{Div-Lem-R4} admits the representation \eqref{Function-Rep}. Furthermore, the fact that
\begin{align*}
\widehat{\frac{\partial^{2} f}{\partial \lambda_{j}\partial \lambda_{k}}}\in L^{1}(\Rn)\quad\text{for all}~1\le j\le k\le n,
\end{align*}
satisfies \eqref{Function-Bound}. This proves \eqref{Div-Lem-R2}.  The proof of \eqref{Div-Lem-R1} follows a similar approach like \eqref{Div-Lem-R2}. This completes the proof of the lemma.
\end{proof}

\subsection{Multiple Operator Integral (MOI)} To proceed further we shall need to utilize
multiple operator integral, which is one of our key tool. We use the following definition of the multiple operator integral, introduced in \cite{PellerMOI}. An interested reader can find a more detailed discussion of MOI in \cite{Pe16-Survey}.

\begin{dfn}\label{MOI-Def}
Let $f\in\mathfrak{G}(\Rn)$ be given by \eqref{Function-Rep}. Let $H_{1},\ldots,H_{n}$ be (not necessarily bounded) self-adjoint operators in $\hil$. The multiple operator integral
\begin{align*}
T_{f}^{H_{1},\ldots,H_{n}}:\bh\times\cdots\times\bh\to\bh
\end{align*}
is defined by
\begin{align*}
T_{f}^{H_{1},\ldots,H_{n}}(V_{1},\ldots,V_{n-1})=\int_{\Omega}a_{1}(H_{1},\sigma)V_{1}\cdots V_{n-1}a_{n}(H_{n},\sigma)\,d\nu(\sigma)
\end{align*}
for $V_{j}\in\bh$, where $a_{j}(\cdot,\cdot)$ and $(\Omega,\nu)$ are taken from the representation \eqref{Function-Rep} and the integral is understood in the sense of Bochner integral. The value $T_{f}^{H_{1},\ldots,H_{n}}(V_{1},\ldots,V_{n-1})$  does not depend on the particular representation of the right-hand side of \eqref{Function-Rep} (see, e.g., \cite[Lemma 3.1]{PellerMOI}).
\end{dfn}

The following observations are direct consequence of Lemma \ref{Div-Lem} and Definition \ref{MOI-Def}.

\begin{lma}\label{MOI-Presentation}
Let $n\in\N$ with $n\ge2$. Let $H_{1},\ldots,H_{n+2}$ be (not necessarily bounded) self-adjoint operators, and let $V_{1},\ldots,V_{n+1}\in\bh$. Then, the following assertions hold:
\begin{enumerate}[{\normalfont(i)}]
\item If $f\in\mathcal{W}_{1}(\Rn)$, then
\begin{align}
\nonumber&T_{f_{j}^{[1]}}^{H_{1},\ldots,H_{n+1}}(V_{1},\ldots,V_{n})\\
\label{MOI-1}=&\frac{i}{(2\pi)^{n/2}}\int_{\Rn}\int_{0}^{t_{j}}\widehat{f}(\vec{t})\bigg[\prod_{l=1}^{j-1}e^{it_{l}H_{l}}V_{l}\bigg]\,e^{i(t_{j}-s)H_{j}}V_{j}e^{isH_{j+1}}\bigg[\prod_{l=j+1}^{n}V_{l}e^{it_{l}H_{l+1}}\bigg]dsd\vec{t},
\end{align}
for all $1\le j\le n$.
\item If $f\in\mathcal{W}_{2}(\Rn)$, then
\begin{align}
\nonumber T_{f_{j}^{[2]}}^{H_{1},\ldots,H_{n+2}}&(V_{1},\ldots,V_{n+1})=\frac{-1}{(2\pi)^{n/2}}\int_{\Rn}\int_{0}^{t_{j}}\int_{0}^{s}\widehat{f}(\vec{t})\bigg[\prod_{l=1}^{j-1}e^{it_{l}H_{l}}V_{l}\bigg]\\
\label{MOI-2}&\times e^{i(t_{j}-s)H_{j}}V_{j}\,e^{i(s-p)H_{j+1}}V_{j+1}e^{ipH_{j+2}}\bigg[\prod_{l=j+2}^{n+1}V_{l}e^{it_{l-1}H_{l+1}}\bigg]dpdsd\vec{t},
\end{align}
for all $1\le j\le n$; and
\begin{align}
\nonumber
T_{f_{j,k}^{[2]}}^{H_{1},\ldots,H_{n+2}}&(V_{1},\ldots,V_{n+1})=\frac{-1}{(2\pi)^{n/2}}\int_{\Rn}\int_{0}^{t_{j}}\int_{0}^{t_{k}}\widehat{f}(\vec{t})\bigg[\prod_{l=1}^{j-1}e^{it_{l}H_{l}}V_{l}\bigg]\,e^{i(t_{j}-s)H_{j}}V_{j}e^{isH_{j+1}}V_{j+1}\\
\label{MOI-3}&\times\bigg[\prod_{l=j+1}^{k-1}e^{it_{l}H_{l+1}}V_{l+1}\bigg]e^{i(t_{k}-p)H_{k+1}}V_{k+1}e^{ipH_{k+2}}\bigg[\prod_{l=k+2}^{n+1}V_{l}e^{it_{l-1}H_{l+1}}\bigg]dpdsd\vec{t},
\end{align}
for all $1\le j<k\le n$.
\end{enumerate}
\end{lma}

The following properties of MOI are crucial in proving our subsequent results.

\begin{thm}\label{MOI-Thm1}
Let $n\in\N$ with $n\ge2$, and let $H_{1},\ldots,H_{n+2}$ be (not necessarily bounded) self-adjoint operators. Then the following assertions hold.
\begin{enumerate}[{\normalfont(i)}]
\item Let $f\in C(\Rn)\cap L^{1}(\Rn)$ with $\widehat{f}\in L^{1}(\Rn)$. Then the mapping $T_{f}^{H_{1},\ldots,H_{n}}:\bh\times\cdots\times\bh\to\bh$ is bounded and 
\begin{align}
\label{MOI-Bound-1}\left\|T_{f}^{H_{1},\ldots,H_{n}}\right\|\le\frac{1}{(2\pi)^{n/2}}~\|\widehat{f}\|_{L^{1}(\Rn)}.
\end{align}
\item Let $f\in\Wn$. Then the mapping $T_{f_{j}^{[1]}}^{H_{1},\ldots,H_{n+1}}:\bh\times\cdots\times\bh\to\bh$ is bounded and 
\begin{align}
\label{MOI-Bound-2}\left\|T_{f_{j}^{[1]}}^{H_{1},\ldots,H_{n+1}}\right\|\le\frac{1}{(2\pi)^{n/2}}~\left\|\widehat{\frac{\partial f}{\partial\lambda_{j}}}\right\|_{L^{1}(\Rn)}\ \ \mbox{for $1\le j\le n$}.
\end{align}
\item Let $f\in\Wnn$. Then the mappings $T_{f_{j,k}^{[2]}}^{H_{1},\ldots,H_{n+2}},T_{f_{j}^{[2]}}^{H_{1},\ldots,H_{n+2}}:\bh\times\cdots\times\bh\to\bh$ are bounded and 
\begin{equation}\label{MOI-Bound-3}
\begin{aligned}
\left\|T_{f_{j,k}^{[2]}}^{H_{1},\ldots,H_{n+2}}\right\|&\le\frac{1}{(2\pi)^{n/2}}~\left\|\widehat{\frac{\partial^{2} f}{\partial\lambda_{j}\partial\lambda_{k}}}\right\|_{L^{1}(\Rn)}\ \ \ \mbox{for $1\le j<k\le n$},\\
\left\|T_{f_{j}^{[2]}}^{H_{1},\ldots,H_{n+2}}\right\|&\le\frac{1}{2\cdot(2\pi)^{n/2}}~\left\|\widehat{\frac{\partial^{2} f}{\partial\lambda_{j}^{2}}}\right\|_{L^{1}(\Rn)}\ \ \ \mbox{for $1\le j\le n$}.
\end{aligned}
\end{equation}
\end{enumerate}
\end{thm}	

\begin{proof}
We only prove \eqref{MOI-Bound-3}. The proofs of \eqref{MOI-Bound-1} and \eqref{MOI-Bound-2} are analogous. Consider $1\le j< k\le n$. Let $f\in\Wnn$. It follows from Lemma \ref{Div-Lem} that $f_{j,k}^{[2]}\in\mathfrak{G}(\R^{n+2})$. Since for any self-adjoint operator $H$ in $\hil$ the map $t\in\R\mapsto e^{itH}$ is strongly continuous, so the Bochner integral \eqref{MOI-3} converges in the strong operator topology. This shows that
\begin{align}
\nonumber\left\|T_{f_{j,k}^{[2]}}^{H_{1},\ldots,H_{n+2}}(V_{1},\ldots,V_{n+1})\right\|&\le\frac{1}{(2\pi)^{n/2}}\int_{\Rn}\left|t_{j}t_{k}\widehat{f}(\vec{t})\right|d\vec{t}\cdot\prod_{l=1}^{n+1}\left\|V_{l}\right\|\\
\nonumber&=\frac{1}{(2\pi)^{n/2}}\left\|\widehat{\frac{\partial^{2} f}{\partial\lambda_{j}\partial\lambda_{k}}}\right\|_{L^{1}(\Rn)}\cdot\prod_{l=1}^{n+1}\left\|V_{l}\right\|.
\end{align}
This further establishes the estimate \eqref{MOI-Bound-3}. Noting that
\[-t_{j}^{2}\widehat{f}(\vec{t})=\widehat{\frac{\partial^{2} f}{\partial\lambda_{j}^{2}}}(\vec{t})\]
the bound for $T_{f_{j}^{[2]}}^{H_{1},\ldots,H_{n+2}}$ can be proved analogously. This completes the proof.
\end{proof}

\begin{crl}\label{MOI-Crl}
Let $n\in\N$ with $n\ge2$, and let $H_{1},\ldots,H_{n+2}$ be bounded self-adjoint operators on $\hil$. Then the following assertions hold.
\begin{enumerate}[{\normalfont(i)}]
\item Let $f\in\Wn$ and $V_{j}\in\il$. Then,
\begin{align}\label{MOI-Crl-R1}
\Big\|T_{f_{j}^{[1]}}^{H_{1},\ldots,H_{n+1}}(I,\ldots,I,\underbrace{V_{j}}_{j},I,\ldots,I)\Big\|_{\il}\le\frac{1}{(2\pi)^{n/2}}~\left\|\widehat{\frac{\partial f}{\partial\lambda_{j}}}\right\|_{L^{1}(\Rn)}\cdot\|V_{j}\|_{\il}
\end{align}
holds for all $1\le j\le n$.
\item Let $f\in\Wnn$ and $V_{j},V_{k}\in\ideal$. Then,
\begin{align}
\nonumber&\Big\|T_{f_{j,k}^{[2]}}^{H_{1},\ldots,H_{n+2}}(I,\ldots,I,\underbrace{V_{j}}_{j},I,\ldots,I,\underbrace{V_{k}}_{k+1},I,\ldots,I)\Big\|_{\il}\\
\label{MOI-Crl-R2}\le&\frac{1}{(2\pi)^{n/2}}~\left\|\widehat{\frac{\partial^{2} f}{\partial\lambda_{j}\partial\lambda_{k}}}\right\|_{L^{1}(\Rn)}\cdot\|V_{j}\|_{\ideal}\|V_{k}\|_{\ideal}
\end{align}
holds for all $1\le j<k\le n$, and 
\begin{align}\label{MOI-Crl-R3}
\Big\|T_{f_{j}^{[2]}}^{H_{1},\ldots,H_{n+2}}(I,\ldots,I,\underbrace{V_{j}}_{j},V_{j},I,\ldots,I)\Big\|_{\il}\le\frac{1}{2\cdot(2\pi)^{n/2}}~\left\|\widehat{\frac{\partial^{2}f}{\partial\lambda_{j}^{2}}}\right\|_{L^{1}(\Rn)}\cdot\|V_{j}\|_{\ideal}^{2}
\end{align}
holds for all $1\le j\le n$.
\end{enumerate}	
In addition, if $\il=\boh$, then \eqref{MOI-Crl-R1}-\eqref{MOI-Crl-R3} hold for not necessarily bounded self-adjoint operators $H_{1},\ldots,H_{n+2}$.
\end{crl}

\begin{proof}
We will only establish \eqref{MOI-Crl-R2}; the estimates \eqref{MOI-Crl-R1} and \eqref{MOI-Crl-R3} can be established completely analogously. For $V_{j},V_{k}\in\ideal$ equation \eqref{MOI-3} yields
\begin{align}
\nonumber&T_{f_{j,k}^{[2]}}^{H_{1},\ldots,H_{n+2}}(I,\ldots,I,\underbrace{V_{j}}_{j},I,\ldots,I,\underbrace{V_{k}}_{k+1},I,\ldots,I)\\
\nonumber=&-\int_{\Rn}\int_{0}^{t_{j}}\int_{0}^{t_{k}}\frac{\widehat{f}(\vec{t})}{(2\pi)^{n/2}}\bigg[\prod_{l=1}^{j-1}e^{it_{l}H_{l}}\bigg]\,e^{i(t_{j}-s)H_{j}}V_{j}e^{isH_{j+1}}\bigg[\prod_{l=j+1}^{k-1}e^{it_{l}H_{l+1}}\bigg]\\
\label{MOI-Crl-R5}&\hspace*{2in}\times e^{i(t_{k}-p)H_{k+1}}V_{k}e^{ipH_{k+2}}\bigg[\prod_{l=k+2}^{n+1}e^{it_{l-1}H_{l+1}}\bigg]dpdsd\vec{t}.
\end{align}
Denote the integrand by $\Gamma(\vec{t},s,p)$. Recall from Hypothesis \ref{Assumption}(b) that $\ideal$ is a symmetrically normed ideal satisfying the H\"older-like inequality \eqref{idealinequality}. Also note that for any two self-adjoint operators $A,B\in\bh$ and any $V\in\ideal$, the map \begin{equation}\label{MOI-Crl-R4}
(s,t)\mapsto e^{i(t-s)A}Ve^{isB}
\end{equation}
is continuous in $\|\cdot\|_{\ideal}$. Combining these observations, we conclude that the map $$(\vec{t},s,p)\in\Pi_{1}\mapsto\Gamma(\vec{t},s,p)\in\il$$
is continuous in the ideal norm $\|\cdot\|_{\il}$, where $\Pi_{1}$ is given by \eqref{Div-Lem-R3}. Consequently, the Bochner integral \eqref{MOI-Crl-R5} converges in $\|\cdot\|_{\il}$. Taking the ideal norm $\|\cdot\|_{\il}$ on both sides of \eqref{MOI-Crl-R5} and applying \eqref{idealinequality} yields the estimate \eqref{MOI-Crl-R2}.
	
In particular, if $\il=\boh$, then the map \eqref{MOI-Crl-R4} is continuous in the Hilbert-Schmidt norm $\|\cdot\|_{2}$, even for not necessarily bounded self-adjoint operators $A,B$. As a result, for $V_{j},V_{k}\in\hils$, the map $(\vec{t},s,p)\in\Pi_{1}\mapsto\Gamma(\vec{t},s,p)\in\boh$ is continuous in the trace class norm $\|\cdot\|_{1}$. Hence, the Bochner integral \eqref{MOI-Crl-R5} converges in the trace class norm $\|\cdot\|_{1}$, and the estimate \eqref{MOI-Crl-R2} follows. This completes the proof.
\end{proof}

The subsequent theorem outlines a formula for perturbation.

\begin{thm}\label{MOI-Thm2}
Let $n\in\N$ with $n\ge2$, and let $f\in\mathcal{W}_{1}(\Rn)$. Let $H_{1},\ldots,H_{n},A,B$ be (not necessarily bounded) self-adjoint operators in $\hil$ such that $A-B\in\bh$. Then, for every $1\le j\le n$,
\begin{align}
\nonumber&T_{f}^{H_{1},\ldots,H_{j-1},A,H_{j+1},\ldots,H_{n}}(I,\ldots,I)-T_{f}^{H_{1},\ldots,H_{j-1},B,H_{j+1},\ldots,H_{n}}(I,\ldots,I)\\
\label{MOI-Pert}&=T_{f_{j}^{[1]}}^{H_{1},\ldots,H_{j-1},A,B,H_{j+1},\ldots,H_{n}}(I,\ldots,I,\underbrace{A-B}_{j},I,\ldots,I).
\end{align}
\end{thm}

\begin{proof}	
We prove \eqref{MOI-Pert} only for $j=1$, as the argument for an arbitrary $j$ is analogous. Observe that,
\begin{align}
\nonumber&T_{f}^{A,H_{2},\ldots,H_{n}}(I,\ldots,I)-T_{f}^{B,H_{2},\ldots,H_{n}}(I,\ldots,I)=\int_{\Rn}\frac{\widehat{f}(\vec{t})}{(2\pi)^{n/2}}\left(e^{it_{1}A}-e^{it_{1}B}\right)e^{it_{2}H_{2}}\cdots e^{it_{n}H_{n}}\,d\vec{t},
\end{align}
which due to Duhamel's formula (see, e.g., \cite[Lemma 5.2]{AzDo09}) further reduces to
\begin{align}
\label{MOI-Thm2-R1}i\int_{\Rn}\int_{0}^{t_{1}}\frac{\widehat{f}(\vec{t})}{(2\pi)^{n/2}}\,e^{i(t_{1}-s)A}(A-B)e^{isB}e^{it_{2}H_{2}}\cdots e^{it_{n}H_{n}}\,dsd\vec{t}.
\end{align}
Finally, using \eqref{MOI-1}, the last expression \eqref{MOI-Thm2-R1} becomes
$$T_{f_{1}^{[1]}}^{A,B,H_{2},\ldots,H_{n}}(A-B,I,\ldots,I).$$
This completes the proof of the theorem.
\end{proof}

Finally, we conclude this section with the following needful notations, which will be used in the subsequent sections.

\begin{Notation}\label{Notation}	
We denote by $\N,\Z,\R$, and $\C$ the sets of natural, integer, real, and complex numbers, respectively. Throughout this paper, we assume $n\in\N$ with $n\ge2$. As recalled in the introduction, $\hil$ denotes a complex separable Hilbert space, $\bh$ the algebra of all bounded linear operators on $\hil$, $\mathcal{B}_{n}(\hil)$ the $n$-th Schatten-von Neumann ideal of compact operators on $\hil$, and $\Tr$ the standard trace on $\boh$.

We also adopt the following notations:
\begin{enumerate}[{\normalfont(i)}]
\item Set $\Hop:=(H_1,\ldots,H_n)$, $\V:=(V_1,\ldots,V_n)$, and define for $t\in[0,1]$ \begin{align}\label{Hnt&Lnt}
\Hop(t):=(H_1+tV_1,\ldots,H_n+tV_n),
\end{align}
where $H_j$ and $V_j$ ($1\le j\le n$) are linear operators in $\hil$.
\item We write $\comn$ for the set of $n$-tuple of pairwise commuting elements, i.e., $(T_{1},\ldots,T_{n})\in\comn$ satisfies $T_{i}T_{j}=T_{j}T_{i}$ for all $1\le i,j\le n.$
\item We denote the coordinates $(t_1,\ldots,t_n)$ by $\vec{t}$, the product measure $dt_1dt_2\cdots dt_n$ by $d\vec{t}$, and the multi-index $(k_1,\ldots, k_n)$ by $\vec{k}$.
\item For $f\in L^{1}(\Rn)$, $\widehat{f}$ denotes its Fourier transform.
\item For any bounded trace $\Tril$ on $\il$, we denote the components of the Jordan decomposition \eqref{Split-Pos-Trace} as $\tau_{k,\il}$ ($1\le k\le 4$) and set
\begin{align}\label{Trace-Sum}
\Wtril:=\tau_{1,\il}+\tau_{2,\il}+\tau_{3,\il}+\tau_{4,\il}.	
\end{align}
\end{enumerate}
\end{Notation}

\section{Krein trace formula for self-adjoint operators}\label{Sec3}
The aim of this section is twofold: first, to prove the Krein trace formula for a tuple of bounded, commuting, self-adjoint operators with perturbations allowed to be in the normed ideals $\il$ of $\bh$; second, to establish the same for a tuple of unbounded, commuting, self-adjoint operators with perturbations allowed to be in $\boh$. Note that proving the result for unbounded operators with $\boh$ perturbations requires a more careful analysis compared to the case of bounded operators with ideal perturbations.

Throughout this section we assume Hypothesis \ref{Assumption}(a), that is, $\il$ is a symmetrically normed ideal of $\bh$ with ideal norm $\|\cdot\|_{\il}$, endowed with a trace $\Tril:\il\to\C$ that is $\|\cdot\|_{\il}$-bounded.

\medskip

The following lemma can be seen as a generalization of \cite[Lemma 5.1]{AzDo09} to multivariable functions.

\begin{lma}\label{Kr-Lem1}
Assume Notation \ref{Notation}. Let $\Hop\in\comn$ be an $n$-tuple of self-adjoint (possibly unbounded) operators in $\hil$. If $f\in C(\Rn)\cap L^{1}(\Rn)$ with $\widehat{f}\in L^{1}(\Rn)$, then
\begin{align}\label{Kr-Lem1-R1}
f(\Hop)=T_{f}^{H_{1},\ldots,H_{n}}(I,\ldots,I).
\end{align}
\end{lma}

\begin{proof}
Theorem \ref{MOI-Thm1} and the assumptions on $f$ ensure that $T_{f}^{H_{1},\ldots,H_{n}}(I,\ldots,I)$ is a bounded linear operator. Now, the representation \eqref{Fourier-Rep} together with the joint spectral theorem (see \cite[Theorem 6.5.1 and Section 6.6.2]{BiSo87}, \cite{Patel}) for $\Hop\in\comn$ implies that, for all $h,k\in\hil$,
\begin{align}
\nonumber\left\la f(\Hop)h\,,\,k\right\ra&=\int_{\Rn}f(\lambda_{1},\ldots,\lambda_{n})d\left\la E(\lambda_{1},\ldots,\lambda_{n})h\,,\,k\right\ra\\
\nonumber&=\int_{\Rn}\left(\int_{\Rn}a_{1}(\lambda_{1},\vec{t})\ldots a_{n}(\lambda_{n},\vec{t})\,d\nu(\vec{t})\right)d\left\la E(\lambda_{1},\ldots,\lambda_{n})h\,,\,k\right\ra,
\end{align}
where the functions $a_{1},\ldots,a_{n}:\R\times\Rn\to\C$ and the measure $d\nu(\vec{t})$ is given by \eqref{Fourier-Rep}. By Fubini's theorem we further obtain
\begin{align}
\nonumber\left\la f(\Hop)h\,,\,k\right\ra&=\int_{\Rn}\left(\int_{\Rn}a_{1}(\lambda_{1},\vec{t})\ldots a_{n}(\lambda_{n},\vec{t})\,d\left\la E(\lambda_{1},\ldots,\lambda_{n})h\,,\,k\right\ra\right)d\nu(\vec{t})\\
\nonumber&=\left\la T_{f}^{H_{1},\ldots,H_{n}}(I,\ldots,I)h\,,\,k\right\ra,~\text{for all}~h,k\in\hil.
\end{align}
This proves \eqref{Kr-Lem1-R1}.
\end{proof}

Below, we express the first order derivative of $t\mapsto f(\Hop(t))$ in terms of MOI. Recall the notation $\Hop(t)$ from \eqref{Hnt&Lnt}.

\begin{lma}\label{Kr-Lem2}
Let $H_{j}\in\bh$ and $V_{j}\in\il$ be self-adjoint operators for $j=1,\ldots,n$. Let $f\in\Wn$ and $\Hop(t)\in\comn$ for all $t\in[0,1]$. Then the following assertions hold.
\begin{enumerate}[{\normalfont(i)}]
\item\label{Kr-Lem2-1} For all $t_{0}\in[0, 1]$, the derivative $\frac{d}{dt}\big|_{t=t_{0}}^{}f(\Hop(t))$ exists in the ideal norm $\|\cdot\|_{\il}$, and
\begin{align*}
&\frac{d}{dt}\bigg|_{t=t_{0}}^{}f(\Hop(t))=\sum_{k=1}^{n}T_{f_{k}^{[1]}}^{H_{1}+t_{0}V_{1},\ldots,H_{k}+t_{0}V_{k},H_{k}+t_{0}V_{k},\ldots,H_{n}+t_{0}V_{n}}(I,\ldots,I,\underbrace{V_{k}}_{k},I,\ldots,I).
\end{align*}
\item\label{Kr-Lem2-2} If $\il=\boh$, then \eqref{Kr-Lem2-1} holds even when $H_{1},\ldots,H_{n}$ are not necessarily bounded self-adjoint operators.
\end{enumerate}
\end{lma}

\begin{proof}
Take $H_{j}\in\bh$ and $V_{j}\in\il$ for $j=1,\ldots,n$. By Lemma \ref{Kr-Lem1} and Theorem \ref{MOI-Thm2} we derive,		
\begin{align}
\nonumber&\frac{f(\Hop(t))-f(\Hop(t_{0}))}{t-t_{0}}\\
\nonumber&=\sum_{k=1}^{n}T_{f_{k}^{[1]}}^{H_{1}+t_{0}V_{1},\ldots,H_{k-1}+t_{0}V_{k-1},H_{k}+tV_{k},H_{k}+t_{0}V_{k},H_{k+1}+tV_{k+1},\ldots,H_{n}+tV_{n}}(I,\ldots,I,\underbrace{V_{k}}_{k},I,\ldots,I)\\
\nonumber&=\sum_{k=1}^{n}\psi_{k}(t)~\text{(say)}.
\end{align}	
From \eqref{MOI-1} we have
\begin{align}
\nonumber&\psi_{k}(t)=\\
\label{Kr-Lem2-R4}&i\int_{\Rn}\int_{0}^{t_{k}}\frac{\widehat{f}(\vec{t})}{(2\pi)^{n/2}}\bigg[\prod_{l=1}^{k-1}e^{it_{l}(H_{l}+t_{0}V_{l})}\bigg]e^{i(t_{k}-s)(H_{k}+tV_{k})}V_{k}e^{is(H_{k}+t_{0}V_{k})}\bigg[\prod_{l=k+1}^{n}e^{it_{l}(H_{l}+tV_{l})}\bigg]dsd\vec{t}.
\end{align}
Denote the integrand by $\Gamma(\vec{t},s,t)$. Observe that
\begin{align}
\label{Kr-Lem2-R3}\lim_{t\to t_{0}}\left\|e^{i(t_{k}-s)(H_{k}+tV_{k})}V_{k}e^{is(H_{k}+t_{0}V_{k})}- e^{i(t_{k}-s)(H_{k}+t_{0}V_{k})}V_{k}e^{is(H_{k}+t_{0}V_{k})}\right\|_{\il}=0,
\end{align} 
and for $k+1\le l\le n$,
\begin{align}
\label{Kr-Lem2-R5}\lim_{t\to t_{0}}\left\|e^{it_{l}(H_{l}+tV_{l})}-e^{it_{l}(H_{l}+t_{0}V_{l})}\right\|=0,
\end{align}
both following from Duhamel's formula. Consequently,
\begin{align}
\label{Kr-Lem2-R6}\lim_{t\to t_{0}}\left\|\Gamma(\vec{t},s,t)-\Gamma(\vec{t},s,t_{0})\right\|_{\il}=0.
\end{align}
Applying the Lebesgue dominated convergence theorem for Bochner integrals to \eqref{Kr-Lem2-R4} gives
\begin{align}
\nonumber&\lim_{t\to t_{0}}\psi_{k}(t)=T_{f_{k}^{[1]}}^{H_{1}+t_{0}V_{1},\ldots,H_{k-1}+t_{0}V_{k-1},H_{k}+t_{0}V_{k},H_{k}+t_{0}V_{k},H_{k+1}+t_{0}V_{k+1},\ldots,H_{n}+t_{0}V_{n}}(I,\ldots,\underbrace{V_{k}}_{k},\ldots,I),
\end{align}
for all $1\le k\le n$, where the limit is taken in the ideal  norm $\|\cdot\|_{\il}$. This establishes \eqref{Kr-Lem2-1}. 
	
\medskip
	
In particular, if $\il=\boh$, then for not necessarily bounded self-adjoint operators $H_{j}$ ($1\le j\le n$), the convergence in \eqref{Kr-Lem2-R5} does not necessarily hold in the operator norm, but it does hold in the strong operator topology. However, since $V_{k}\in\boh$, the convergence in \eqref{Kr-Lem2-R3} and \eqref{Kr-Lem2-R6} hold in the trace class norm $\|\cdot\|_{1}$. Therefore, by the same method as above, we obtain
$$\lim_{t\to t_{0}}\left\|\psi_{k}(t)-\psi_{k}(t_{0})\right\|_{1}=0.$$ 
This establishes \eqref{Kr-Lem2-2}, thereby completing the proof of the lemma.
\end{proof}

The proof of the existence of the spectral shift measures and the associated trace formula relies on the following two key lemmas.

\begin{lma}\label{Kr-Lem3}
Let $H_{j}\in\bh$ and $V_{j}\in\il$ be self-adjoint operators for $j=1,\ldots,n$. Let $f\in\Wn$ and $\Hop(t)\in\comn$ for all $t\in[0,1]$. Then for all $t_{0}\in[0,1],$
\begin{align}
\label{Kr-Lem3-R1}\Tril\left(\frac{d}{dt}\bigg|_{t=t_{0}}^{}f(\Hop(t))\right)=\frac{d}{dt}\bigg|_{t=t_{0}}^{}\Tril\left(f(\Hop(t))-f(\Hop)\right).
\end{align}
In particular, if $\il=\boh$, then $H_{j}$ can be chosen as not necessarily bounded self-adjoint operators.
\end{lma}

\begin{proof}
Note that Theorem \ref{MOI-Thm2}, \eqref{MOI-Crl-R1}, and the assumption on $f$ ensure that $f(\Hop(t))-f(\Hop(t_{0}))\in\il$ for all $t,t_{0}\in[0,1]$. Recall from Hypothesis \ref{Assumption}(a) that the trace $\Tril$ is $\|\cdot\|_{\il}$-bounded. Therefore, by the $\|\cdot\|_{\il}$-boundedness of $\Tril$ and Lemma \ref{Kr-Lem2}, we establish \eqref{Kr-Lem3-R1}.
\end{proof}

\begin{lma}\label{Kr-Lem4}
Let $n, N\in\N$. Let $\{\delta_{i}\}_{i=1}^{N}$ be a partition of $\C^{n}$, and let $E(\cdot)$ be a spectral measure on $\C^{n}$. Let $\Tril$ be a $\|\cdot\|_{\il}$-bounded trace on $\il$ with the Jordan decomposition given by $\Tril=\tau_{1,\il}-\tau_{2,\il}+i\tau_{3,\il}-i\tau_{4,\il}$, where $\tau_{k,\il}$ $(1\le k\le 4)$ are the positive components of $\Tril$. If $V\in\il,$ then
$$\sum_{i=1}^{N}\left|\Tril(E(\delta_{i})V)\right|\le\Wtril(|\Re(V)|+|\Im(V)|),$$
where $\Wtril=\tau_{1,\il}+\tau_{2,\il}+\tau_{3,\il}+\tau_{4,\il}$.
\end{lma}
	
\begin{proof}
The proof follows from \cite[Lemma 3.9]{DySk14}.
\end{proof}

We adhere to the following notation:
\[\Wtril=\tau_{1,\il}+\tau_{2,\il}+\tau_{3,\il}+\tau_{4,\il},\]
where $\tau_{k,\il}\,(1\le k\le 4)$ are the positive components of the Jordan decomposition of $\Tril$.

The following theorem shows that the trace of the derivative of a multivariable operator function can be written via partial derivatives of $f$.

\begin{thm}\label{Kr-Thm1}
Assume Notation \ref{Notation}. Let $H_{j}\in\bh$ and $V_{j}\in\il$ be self-adjoint operators for $j=1,\ldots,n$. Let $\Tril$ be a $\|\cdot\|_{\il}$-bounded trace on $\il$ with the Jordan decomposition given by $\Tril=\tau_{1,\il}-\tau_{2,\il}+i\tau_{3,\il}-i\tau_{4,\il}$, where $\tau_{k,\il}$ $(1\le k\le 4)$ are the positive components of $\Tril$. Let $f\in\Wn$ and $\Hop(t)\in\comn$ for all $t\in[0,1]$. Then for all $t_{0}\in[0,1],$
\begin{align}
\label{Kr-Thm1-R1}&\Tril\left(\frac{d}{dt}\bigg|_{t=t_{0}}^{}f(\Hop(t))\right)=\sum_{j=1}^{n}\Tril\left(\frac{\partial f}{\partial \lambda_{j}}(\Hop(t_{0}))V_{j}\right),
\end{align}
and
\begin{align}
\label{Kr-Thm1-R2}&\left|\Tril\left(\frac{\partial f}{\partial \lambda_{j}}(\Hop(t_{0}))V_{j}\right)\right|\le\Wtril\left(|V_{j}|\right)\left\|\frac{\partial f}{\partial \lambda_{j}}\right\|_{L^{\infty}(\Rn)}.
\end{align}
In particular, if $\il=\boh$, then $H_{j}~(1\le j\le n)$ can be chosen as not necessarily bounded self-adjoint operators.
\end{thm}

\begin{proof}
Consider $H_{j}\in\bh$ and $V_{j}\in\il$ for $j=1,\ldots,n$. Lemma \ref{Kr-Lem2} and \eqref{MOI-1} yield
\begin{align}
\nonumber&\Tril\left(\frac{d}{dt}\bigg|_{t=t_{0}}^{}f(\Hop(t))\right)=\sum_{j=1}^{n}\frac{i}{(2\pi)^{n/2}}\left\{\int_{\Rn}\int_{0}^{t_{j}}\Tril\left(\Gamma_{j}(\vec{t},s,t_{0})\right)\widehat{f}(\vec{t})dsd\vec{t}\,\right\},
\end{align}
where 
\begin{align*}
\Gamma_{j}(\vec{t},s,t_{0})=\bigg(\prod_{l=1}^{j-1}e^{it_{l}(H_{l}+t_{0}V_{l})}\bigg)e^{i(t_{j}-s)(H_{j}+t_{0}V_{j})}V_{j}e^{is(H_{j}+t_{0}V_{j})}\bigg(\prod_{l=j+1}^{n}e^{it_{l}(H_{l}+t_{0}V_{l})}\bigg).
\end{align*}
Now the cyclicity of the trace $\Tril$ and the pairwise commutativity of $H_{1}(t_{0}),\ldots,H_{n}(t_{0})$ implies
\begin{align}
\nonumber\Tril\left(\frac{d}{dt}\bigg|_{t=t_{0}}^{}f(\Hop(t))\right)&=\frac{i}{(2\pi)^{n/2}}\left\{\sum_{j=1}^{n}\int_{\Rn}\int_{0}^{t_{j}}\Tril\left(\bigg(\prod_{l=1}^{n}e^{it_{l}\,(H_{l}+t_{0}V_{l})}\bigg)\,V_{j}\right)\,\widehat{f}(\vec{t})\,ds\,d\vec{t}\,\right\}\\
\nonumber&=\frac{1}{(2\pi)^{n/2}}\Tril \left\{\sum_{j=1}^{n}\int_{\Rn}\widehat{\frac{\partial f}{\partial \lambda_{j}}}(\vec{t})\,\bigg(\prod_{l=1}^{n}e^{it_{l}\,(H_{l}+t_{0}V_{l})}\bigg)\,d\vec{t}\,\,\,V_{j}\right\}.
\end{align}
Further, it follows from Lemma \ref{Kr-Lem1} that
$$\frac{1}{(2\pi)^{n/2}}\Tril\left\{\sum_{j=1}^{n}\int_{\Rn}\widehat{\frac{\partial f}{\partial \lambda_{j}}}(\vec{t})\,\bigg(\prod_{l=1}^{n}e^{it_{l}\,(H_{l}+t_{0}V_{l})}\bigg)\,d\vec{t}\,\,V_{j}\right\}=\sum_{j=1}^{n}\Tril\left(\frac{\partial f}{\partial \lambda_{j}}(\Hop(t_{0}))V_{j}\right).$$
This proves \eqref{Kr-Thm1-R1}.
	
Since $\Hop(t_{0})\in\comn$ is a tuple of bounded self-adjoint operators, the joint spectral theorem yields
\begin{align*}
&\frac{\partial f}{\partial \lambda_{j}}(\Hop(t_{0}))=\int_{\sigma(\Hop(t_{0}))}\frac{\partial f}{\partial \lambda_{j}}(\lambda_{1},\ldots,\lambda_{n})\,dE_{t_{0}}(\lambda_{1},\ldots,\lambda_{n}),
\end{align*}
where $E_{t_{0}}(\cdot)$ is the joint spectral measure of $\Hop(t_{0})$. Then there exists a sequence of Borel partition $(\delta_{m,k})_{1\leq k\leq m}$ of $\R^n$ and a sequence of $n$-tuple of real numbers $(\bm \lambda_{m,k})_{1\le{k}\le m}$ such that
\begin{align}
\label{Kr-Thm1-R4}\Tril\left(\frac{\partial f}{\partial \lambda_{j}}(\Hop(t_{0}))V_{j}\right) =\lim_{m\to\infty}\sum_{1\le k\le m}\frac{\partial f}{\partial \lambda_{j}}(\bm \lambda_{m,k})\,\Tril\left(E_{t_{0}}(\delta_{m,k})V_{j}\right).
\end{align}
Using the representation \eqref{Kr-Thm1-R4} along with Lemma \ref{Kr-Lem4}, we immediately obtain the following estimate:
\begin{align}
\nonumber&\left|\Tril\left(\frac{\partial f}{\partial \lambda_{j}}(\Hop(t_{0}))V_{j}\right)\right|\le\Wtril\left(|V_{j}|\right)\left\|\frac{\partial f}{\partial \lambda_{j}}\right\|_{L^{\infty}(\Rn)}.
\end{align}
This completes the proof for bounded $H_{j}$ $(1\le j\le n)$ with arbitrary ideal perturbations. The case of not necessarily bounded $H_{j}$ with a perturbation in the trace class ideal follows a similar argument. This concludes the proof of the theorem.
\end{proof}

The following corollary represents an analog of the fundamental theorem of calculus.

\begin{crl}\label{Kr-Cor1}
Let $H_{j}\in\bh$ and $V_{j}\in\il$ be self-adjoint operators for $j=1,\ldots,n$. Let $\Hop(t)\in\comn$ for all $t\in[0,1]$. Then for every $f\in\Wnn$,
\begin{align}
\label{Kr-Cor1-R1}&\Tril\big\{f(\Hop(1))-f(\Hop)\big\}=\int_{0}^{1}\frac{d}{dt}\Tril\left(f(\Hop(t))-f(\Hop)\right)dt=\int_{0}^{1}\Tril\left(\frac{d}{dt}f(\Hop(t))\right)dt.
\end{align}
In particular, if $\il=\boh$, then $H_{j}~(1\le j\le n)$ can be chosen as not necessarily bounded self-adjoint operators.
\end{crl}

\begin{proof}
We provide a proof assuming $H_{j}\in\bh$ and $V_{j}\in\il$ for $j=1,\ldots,n$. The proof for not necessarily bounded $H_{j}$ with $V_{j}\in\boh$ is analogous. Note that the continuity of the function
\begin{align}
\label{Kr-Cor1-R2}&\psi:t\in[0,1]\mapsto\frac{d}{dt}\Tril\left(f(\Hop(t))-f(\Hop)\right)
\end{align}
for every $t\in[0, 1]$ suffices to establish \eqref{Kr-Cor1-R1}. Lemma \ref{Kr-Lem2} ensures that
$$\frac{d}{ds}\bigg|_{s=t}f(\Hop(s))\in\il,~\text{for all}~t\in[0,1].$$
Applying Lemma \ref{Kr-Lem3} and Theorem \ref{Kr-Thm1}, we obtain
\begin{align}
\nonumber\psi(t+\epsilon)-\psi(t)&=\Tril\left(\frac{d}{ds}\bigg|_{s=t+\epsilon}^{}f(\Hop(s))-\frac{d}{ds}\bigg|_{s=t}^{}f(\Hop(s))\right)\\
\label{Kr-Cor1-R3}&=\Tril\left(\sum_{j=1}^{n}\frac{\partial f}{\partial \lambda_{j}}(\Hop(t+\epsilon))V_{j}-\sum_{j=1}^{n}\frac{\partial f}{\partial \lambda_{j}}(\Hop(t))V_{j}\right).
\end{align}
Denote $g:=\frac{\partial f}{\partial \lambda_{1}}$ and consider the case $j=1$ in \eqref{Kr-Cor1-R3}. Using Theorem \ref{MOI-Thm2} we deduce that
\begin{align}
\nonumber&g(\Hop(t+\epsilon))-g(\Hop(t))\\
\label{Kr-Cor1-R4}&=\bigg[\sum_{k=1}^{n}T_{g_{k}^{[1]}}^{H_{1}+tV_{1},\ldots,H_{k-1}+tV_{k-1},H_{k}+(t+\epsilon)V_{k},H_{k}+tV_{k},H_{k+1}+(t+\epsilon)V_{k+1},\ldots,H_{n}+(t+\epsilon)V_{n}}(I,\ldots,\underbrace{\epsilon V_{k}}_{k},\ldots,I)\bigg].
\end{align}
Now the symmetric property of the normed ideal $\il$ ensures the following estimate
\begin{align}
\nonumber&\left|\Tril\left\{g(\Hop(t+\epsilon))V_{1}-g(\Hop(t))V_{1}\right\}\right|\\
\label{Kr-Cor1-R5}\le&\|\Tril\|_{\il^{*}}\cdot\|V_{1}\|\cdot\left\|g(\Hop(t+\epsilon))-g(\Hop(t))\right\|_{\il}.
\end{align}
Thus, by \eqref{Kr-Cor1-R4} and Corollary \ref{MOI-Crl}, the last expression \eqref{Kr-Cor1-R5} becomes less than or equal to
\begin{align}
\label{Kr-Cor1-R6}\|\Tril\|_{\il^{*}}\cdot\epsilon\,\,\|V_{1}\|\cdot\left(\sum_{k=1}^{n}\|V_{k}\|_{\il}\,\,\frac{1}{(2\pi)^{n/2}}\left\|\widehat{\frac{\partial^{2} f}{\partial\lambda_{k}\partial\lambda_{1}}}\right\|_{L^{1}(\Rn)}\right).
\end{align}
Similar to \eqref{Kr-Cor1-R6}, one can derive the estimate for $2\le j\le n$. Hence the continuity of the function $\psi(t)$ in \eqref{Kr-Cor1-R2} is achieved. Combining this with Lemma \ref{Kr-Lem3} establishes \eqref{Kr-Cor1-R1}.
\end{proof}

Our main theorem in this section is the following.

\begin{thm}\label{Kr-Thm2}
Assume Notation \ref{Notation}. Let $H_{j}\in\bh$ and $V_{j}\in\il$ be self-adjoint operators for $j=1,\ldots,n$. Let $\Tril$ be a $\|\cdot\|_{\il}$-bounded trace on $\il$ with the Jordan decomposition given by $\Tril=\tau_{1,\il}-\tau_{2,\il}+i\tau_{3,\il}-i\tau_{4,\il}$, where $\tau_{k,\il}$ $(1\le k\le 4)$ are the positive components of $\Tril$. If $\Hop(t)\in\comn$ for all $t\in[0,1]$, then there exist finite measures $\mu_{1},\ldots,\mu_{n}$ on $\Rn$ such that for every $1\le j\le n,$
\begin{align}
\label{Kr-Thm2-R1}&\|\mu_{j}\|\le\Wtril(|V_{j}|),
\end{align}
and
\begin{align}
\label{Kr-Thm2-R2}&\Tril\left\{f(\Hop(1))-f(\Hop)\right\}=\sum_{j=1}^{n}\int_{\Rn}\frac{\partial f}{\partial \lambda_{j}}(\lambda_{1},\ldots,\lambda_{n})\,d\mu_{j}(\lambda_{1},\ldots,\lambda_{n}),
\end{align}
for every $f\in\Wnn$.
\end{thm}

\begin{proof}
Combining the results of Corollary \ref{Kr-Cor1} and Theorem \ref{Kr-Thm1}, we obtain
\begin{align}
\label{Kr-Thm2-R3}&\Tril\left\{f(\Hop(1))-f(\Hop)\right\}=\sum_{j=1}^{n}\int_{0}^{1}\Tril\left(\frac{\partial f}{\partial \lambda_{j}}(\Hop(t))V_{j}\right)dt.
\end{align}
Consider the following collection of linear functionals
\begin{align}
\label{Kr-Thm2-R4}&\varphi_{j}:~\text{span}\left\{\frac{\partial f}{\partial \lambda_{j}}:\,\,\,f\in\Wnn\right\}\to\C,
\end{align}
defined by
\begin{align}
\label{Kr-Thm2-R5}&\varphi_{j}\left(\frac{\partial f}{\partial \lambda_{j}}\right)=\int_{0}^{1}\Tril\left(\frac{\partial f}{\partial \lambda_{j}}(\Hop(t))V_{j}\right)dt,\quad\text{for}~1\le j\le n.
\end{align}
Consequently, it follows from \eqref{Kr-Thm1-R2} that
\begin{align*}
\left|\varphi_{j}\left(\frac{\partial f}{\partial \lambda_{j}}\right)\right|\le\Wtril(|V_{j}|)\left\|\frac{\partial f}{\partial \lambda_{j}}\right\|_{L^{\infty}(\Rn)},\quad\text{for}~1\le j \le n,
\end{align*}
where $\Wtril$ is given by \eqref{Trace-Sum}. By the Hahn-Banach theorem and the Riesz representation theorem for elements of $(C_0(\Rn))^*$,  there exists a complex measure $\mu_{j}$ corresponding to the bounded linear functional $\varphi_{j}$ in \eqref{Kr-Thm2-R4} such that
$$\|\varphi_{j}\|_{\infty}=\|\mu_{j}\|\le\Wtril(|V_{j}|),$$
and
\begin{align}
\label{Kr-Thm2-R7}\varphi_{j}\left(\frac{\partial f}{\partial \lambda_{j}}\right)=\int_{\Rn}\frac{\partial f}{\partial \lambda_{j}}(\lambda_{1},\ldots,\lambda_{n})\,d\mu_{j}(\lambda_{1},\ldots,\lambda_{n}).
\end{align}
Finally, combining \eqref{Kr-Thm2-R7} with \eqref{Kr-Thm2-R5} and \eqref{Kr-Thm2-R3} establishes \eqref{Kr-Thm2-R2}. This completes the proof.
\end{proof}

The next theorem relaxes the commutativity assumption imposed on the tuple $\Hop(t)$ for $t\in(0,1)$ in Theorem \ref{Kr-Thm2} for the Lorentz ideal $\il=\mathcal{M}_\psi$.

\begin{thm}\label{Kr-Thm3}
Assume Notation \ref{Notation}. Let $\mathcal{M}_\psi$ be a Lorentz ideal, where the function $\psi$ satisfies \eqref{psi-growth-condition} for some $0<\varepsilon<1/2$. Let $\tau_{\psi}$ be a bounded singular trace on $\mathcal{M}_\psi$. Let $H_{j}\in\bh$ and $V_{j}\in\mathcal{M}_\psi$ be self-adjoint operators for $j=1,\ldots,n$. If $\Hop,\Hop(1)\in\comn$, then there exist finite measures $\mu_{1},\ldots,\mu_{n}$ on $\Rn$ such that
\begin{align}
\label{Kr-Thm3-R3}&\tau_\psi\left\{f(\Hop(1))-f(\Hop)\right\}=\sum_{j=1}^{n}\int_{\Rn}\frac{\partial f}{\partial \lambda_{j}}(\lambda_{1},\ldots,\lambda_{n})\,d\mu_{j}(\lambda_{1},\ldots,\lambda_{n}),
\end{align}
for every $f\in\W_{2}(\Rn)$.
\end{thm}

\begin{proof}
Consider $f\in\W_{2}(\Rn)$. By Theorem \ref{MOI-Thm2} and Lemma \ref{Kr-Lem2}, one can verify that 
$$\left(f(\Hop(1))-f(\Hop)-\frac{d}{dt}\bigg|_{t=0}^{}T_f^{H_1(t),\ldots,H_n(t)}(I,\ldots,I)\right)\in\mathcal{M}_\psi^{2}.$$
Using \eqref{Dixmier-Property}, the trace of the above object vanishes. Therefore,
$$\tau_{\psi}\left(f(\Hop(1))-f(\Hop)\right)=\tau_{\psi}\left(\frac{d}{dt}\bigg|_{t=0}^{}T_f^{H_1(t),\ldots,H_n(t)}(I,\ldots,I)\right).$$
Since $\Hop\in\comn$, Theorem \ref{Kr-Thm1} further implies	
\begin{align}
\nonumber&\tau_\psi\left\{f(\Hop(1))-f(\Hop)\right\}=\sum_{j=1}^{n}\tau_\psi\left(\frac{\partial f}{\partial \lambda_{j}}(\Hop)V_{j}\right).
\end{align}
Finally, by reasoning analogous to the proof of Theorem \ref{Kr-Thm2}, we conclude the proof.
\end{proof}

The following theorem establishes a result analogous to Theorem \ref{Kr-Thm2}, but relaxes the boundedness assumption on the tuple $\Hop(t)$ for $t\in[0, 1]$, in the case $\il=\boh$.

\begin{thm}\label{Kr-Thm4}
Assume Notation \ref{Notation}. Let $V_{j}\in\boh$ and $H_{j}$ be not necessarily bounded self-adjoint operators in $\hil$ for $j=1,\ldots,n$. If $\Hop(t)\in\comn$ for all $t\in[0,1]$, then there exist measures $\mu_{1},\ldots,\mu_{n}$ on $\Rn$ such that for every $1\le j\le n,$
\begin{align}
\label{Kr-Thm4-R1}&\|\mu_{j}\|\le \|V_{j}\|_{1},
\end{align}
and
\begin{align*}
&\Tr\left\{f(\Hop(1))-f(\Hop)\right\}=\sum_{j=1}^{n}\int_{\Rn}\frac{\partial f}{\partial \lambda_{j}}(\lambda_{1},\ldots,\lambda_{n})\,d\mu_{j}(\lambda_{1},\ldots,\lambda_{n}),
\end{align*}
for every $f\in\Wnn$.
\end{thm}

\begin{proof}
The proof follows closely that of Theorem \ref{Kr-Thm2}.
\end{proof}

Before proceeding further, we summarize our results on the Krein trace formula.

\begin{rmrk} 
Assume Hypothesis \ref{Assumption}(a). Let $H_{j}$ be self-adjoint operators in $\hil$ and $V_{j}=V_{j}^*\in\il$ for $j=1,\ldots,n$. Then the following results have been established: 
\begin{enumerate}[{\normalfont(i)}]
\item If $H_{j}\in\bh$ for $1\le j\le n$, then Theorem \ref{Kr-Thm2} yields formula \eqref{PellerProblem} under an extra assumption that $\V$ is commutative (see Lemma \ref{Sufficient-cond-Comn} below). In the special case of the Lorentz ideal $\il=\mathcal{M}_{\psi}$ (with $\psi$ satisfying \eqref{psi-growth-condition} for some $0< \varepsilon<1/2$), Theorem \ref{Kr-Thm3} establishes \eqref{PellerProblem} without this extra commutativity assumption.

\item If $\il=\boh$, then Theorem \ref{Kr-Thm4} establishes formula \eqref{PellerProblem} under the additional assumption that $\V$ is commutative. Here the initial operator tuple $\Hop$ may be unbounded.
\end{enumerate}
\end{rmrk}

\begin{ppsn}\label{KrPrp1}
Let $H_{j}$ be self-adjoint operators in $\hil$ and $V_{j}=V_{j}^*\in\bh$ for $j=1,\ldots,n$. Suppose $\Hop(t)\in\comn$ for all $t\in[0,1]$. Then the following assertions hold:
\begin{enumerate}[{\normalfont(i)}]
\item If $H_{j}\in\bh$ and $V_{j}\in\il$ satisfy $\Wtril(|V_{j}|)=0$ for $1\le j\le n$, then Theorem \ref{Kr-Thm2} holds with $\mu_{j}=0$ for $1\le j\le n$.
\item If $V_{j}\in\boh$ satisfies $\Tr(|V_{j}|)=0$ for $1\le j\le n$, then Theorem \ref{Kr-Thm4} holds with $\mu_{j}=0$ for $1\le j\le n$.
\end{enumerate} 
\end{ppsn}

\begin{proof}
Based on the estimates \eqref{Kr-Thm2-R1} and \eqref{Kr-Thm4-R1} for the total variation of measures $\mu_{j}$, $1\le j\le n$, the result follows immediately. 
\end{proof}

We now provide several sufficient conditions for the tuple $\Hop(t)$ to be commutative for all $t\in[0,1]$.

\begin{lma}{\normalfont\cite[Lemma 3.7]{Sk15}}\label{Sufficient-cond-Comn}
Assume Notation \ref{Notation}. The following statements are equivalent.
\begin{enumerate}[\normalfont(i)]
\item $\Hop(t):=(H_{1}+tV_{1},\ldots,H_{n}+tV_{n})\in\comn$, for every $t\in[0,1]$.
\item $\Hop,\Hop(1),\mathbf{V}_{n}\in\comn$.
\item $\Hop,\mathbf{V}_{n}\in\comn$, $[H_{i},V_{j}]=[H_{j},V_{i}],$ for all $1\le i,j\le n$.
\end{enumerate}
\end{lma}

To clarify the conditions of the above lemma, it is necessary to consider concrete examples:

\begin{xmpl}[Bounded case]\label{Exp-1}
\begin{enumerate}[{\normalfont(i)}]
\item Let $H_0,V\in\bh$ be two self-adjoint operators with $H_0V=VH_0.$ Let $f_i\in\mathcal{W}_1(\R)$, $1\le i\le n$, be real-valued functions. Define
\begin{align*}
H_i=f_i(H_0)\ \ \ \mbox{and}\ \ \  
V_i=f_{i}(H_0+V)-f_{i}(H_0).
\end{align*}
Then, by \cite[Lemma 3.7]{DySk14}, $V_i\in\il$ (resp. $V_{i}\in\ideal$) whenever $V\in\il$ (resp. $V\in\ideal$). Consequently, by the spectral theorem and the functional calculus for self‑adjoint operators, the tuples $\Hop,\Hop(1),\V\in\comn$.

\medskip
		
\item Let $(A_1,\ldots,A_n)$, $(B_1,\ldots,B_n)$ be two $n$-tuple of bounded self-adjoint operators on $\hil$. For $1\le k\le n$, define
\begin{align*}
H_{k}&:=I\oplus\cdots\oplus I\oplus \underbrace{A_{k}}_{k}\oplus I\oplus\cdots\oplus I,\\
V_{k}&:=0\oplus\cdots\oplus 0\oplus \underbrace{B_{k}}_{k}\oplus 0\oplus\cdots\oplus 0.
\end{align*} 
Straightforward verification shows that the self-adjoint tuple $\Hop(t)\in\comn$ for all $t\in[0,1]$.

\medskip
		
\item Let $\mathcal{K}$ be a finite dimensional subspace of $\hil$, and let $P_{\mathcal{K}}$ denote the orthogonal projection from $\mathcal{H}$ onto $\mathcal{K}$. Let $(A_{1},\ldots,A_{n})$ and $(B_{1},\ldots,B_{n})$ be two $n$-tuple of bounded self-adjoint operators on $\hil$. Let $(C_{1},\ldots,C_{n})\in\comn$ be an $n$-tuple of bounded self-adjoint operators. For $1\leq j\leq n$, define
\begin{align}
\nonumber H_{j}=P_{\mathcal{K}}\otimes\cdots\otimes P_{\mathcal{K}}\otimes \underbrace{P_{\mathcal{K}}A_{j}P_{\mathcal{K}}}_{j}\otimes P_{\mathcal{K}}\otimes\cdots\otimes P_{\mathcal{K}}\otimes C_{j},
\end{align}
and 
\begin{align}
\nonumber V_{j}=P_{\mathcal{K}}\otimes\cdots\otimes P_{\mathcal{K}}\otimes \underbrace{P_{\mathcal{K}}B_{j}P_{\mathcal{K}}}_{j}\otimes P_{\mathcal{K}}\otimes\cdots\otimes P_{\mathcal{K}}\otimes C_{j}.
\end{align}
Then the self-adjoint tuple $\Hop(t)\in\comn$ for all $t\in[0,1]$.
\end{enumerate}	
\end{xmpl}

\begin{xmpl}[Unbounded case]
Let $n=2$. Let $\{e_k\}_{k\in\mathbb{{N}}}$ be an orthonormal basis of $\hil$. Define two self-adjoint operators $H_0, V$ on $\{e_k\}_{k\in\mathbb{{N}}}$ by
\begin{align*}
H_0(e_{2k-1})&= ke_{2k-1}+e_{2k},~ H_0(e_{2k})= e_{2k-1}-ke_{2k},~\text{for}~k\in\N,\text{ and } \\
V(e_1)&=e_2,~V(e_2)=e_1, V(e_k)=0,~\text{ for } k\ge 3.
\end{align*}
Then $H_0$ is an unbounded self-adjoint operator that does not commute with $V$. Set $H_1=H_0$, $H_2=H_0^2$, $V_1=V$, $V_2=V^2$. It is now straightforward to verify that $(H_{1}+V_{1},H_{2}+V_{2})\in\comn$. Consequently, $(H_{1}+tV_{1},H_{2}+tV_{2})\in\comn$ for all $t\in[0,1]$.
\end{xmpl}

\section{Koplienko trace formula for self-adjoint operators}\label{Sec4}
In this section, we establish the Koplienko trace formula for an $n$-tuple of commuting self-adjoint operators (possibly unbounded) with perturbations in the root ideal $\ideal$ of a symmetrically normed ideal $\il$ (which, of course, satisfies Hypothesis \ref{Assumption}) of $\bh$  corresponding to the function space $\mathfrak{R}(\Rn)$ (see \eqref{rational}). Our goal will be accomplished in Theorem \ref{Kp-Thm2}. We begin with the following convention.

\medskip

\noindent\textbf{Convention}: Throughout this section we always deal with possibly unbounded self-adjoint
operators. We will not repeatedly use the phrase \enquote{$H$ is a not necessarily bounded self-adjoint operator in $\hil$}. Instead, we simply write \enquote{$H$ is a self-adjoint operator in $\hil$}.

\medskip

Before we proceed, we fix some notation. Set $H_{l}(t)=H_{l}+tV_{l}$ for $1\le l\le n$ and $t\in[0,1]$, and therefore from \eqref{Hnt&Lnt} we now have:
\begin{align*}
\Hop(t)=(H_{1}(t),\ldots,H_{n}(t)).
\end{align*}
Denote
\begin{align}\label{Rational-without-span}
\mathfrak{R}:=\left\{(\lambda_{1},\ldots,\lambda_{n})\in\Rn\mapsto(z-\lambda_{1})^{-k_{1}}\cdots(z-\lambda_{n})^{-k_{n}}:\,\,\Im(z)\neq 0,\,k_{i}\in\N\right\}.
\end{align}
For $f\in\mathfrak{R}$, the corresponding operator function associated with an $n$-tuple of unbounded self-adjoint operators $\Hop$ (not necessarily commutative) is defined by
$$f(\Hop)=(zI-H_{1})^{-k_{1}}\cdots(zI-H_{n})^{-k_{n}}.$$

The derivative of $f\in\mathfrak{R}$ along the direction $\V$ can be computed via partial derivatives along directions $V_{1},\ldots,V_{n}.$

\begin{lma}\label{Kp-Lem1}
Assume Notation \ref{Notation}. Let $f\in\mathfrak{R}$ be given by
\[f(\lambda_{1},\ldots,\lambda_{n})=(z-\lambda_{1})^{-k_{1}}\cdots(z-\lambda_{n})^{-k_{n}},\]
where $\Im(z)\neq0$ and $k_{1},\ldots,k_{n}\in\N$. Let $H_{j}$ be self-adjoint operators in $\hil$ and $V_{j}=V_{j}^*\in\bh$ for $j=1,\ldots,n$. Then the following assertions hold.
\begin{enumerate}[{\normalfont(i)}]
\item If $V_{j}\in\il$ for $1\le j\le n$, then $t\mapsto\dds\bigg|_{s=t}^{}f(\Hop(s))$ exists in the ideal norm $\|\cdot\|_{\il}$ such that
\begin{flalign}
\label{Kp-Lem1-R1}&\frac{d}{ds}\bigg|_{s=t}^{}f(\Hop(s))=\sum_{j=1}^{n}D_{H_{j}}^{f}(t),
\end{flalign}
where
\begin{align}
\label{Kp-Lem1-R2}&D_{H_{j}}^{f}(t)=\Big[\prod_{l=1}^{j-1}(zI-H_{l}(t))^{-k_{l}}\Big]\frac{d}{ds}\bigg|_{s=t}^{}(zI-H_{j}(s))^{-k_{j}}\Big[\prod_{l=j+1}^{n}(zI-H_{l}(t))^{-k_{l}}\Big].
\end{align}
\item If $V_{j}\in\ideal$ for $1\le j\le n$, then $t\mapsto\ddds\bigg|_{s=t}^{}f(\Hop(s))$ exists in the ideal norm $\|\cdot\|_{\il}$ such that
\begin{flalign}
\label{Kp-Lem1-R3}&\ddds\bigg|_{s=t}^{}f(\Hop(s))=2\sum_{1\le i<j\le n}D_{H_{i},H_{j}}^{f}(t)+\sum_{1\le j\le n}D_{H_{j},H_{j}}^{f}(t),
\end{flalign}
where
\begin{align}
\nonumber& D_{H_{i},H_{j}}^{f}(t)=\Big[\prod_{l=1}^{i-1}(zI-H_{l}(t))^{-k_{l}}\Big]\frac{d}{ds}\bigg|_{s=t}^{}(zI-H_{i}(s))^{-k_{i}}\Big[\prod_{l=i+1}^{j-1}(zI-H_{l}(t))^{-k_{l}}\Big]\\
\label{Kp-Lem1-R4}&\hspace{2.1in}\times \frac{d}{ds}\bigg|_{s=t}^{}(zI-H_{j}(s))^{-k_{j}}\Big[\prod_{l=j+1}^{n}(zI-H_{l}(t))^{-k_{l}}\Big],
\end{align}
and
\begin{align}
\label{Kp-Lem1-R5}&D_{H_{j},H_{j}}^{f}(t)=\Big[\prod_{l=1}^{j-1}(zI-H_{l}(t))^{-k_{l}}\Big]\ddds\bigg|_{s=t}^{}(zI-H_{j}(s))^{-k_{j}}\Big[\prod_{l=j+1}^{n}(zI-H_{l}(t))^{-k_{l}}\Big].
\end{align}
\end{enumerate}
\end{lma}

\begin{proof}
We shall only establish the existence of $t\mapsto\dds\bigg|_{s=t}^{}f(\Hop(s))$ in the ideal norm $\|\cdot\|_{\il}$. The result for the existence of $t\mapsto\ddds\bigg|_{s=t}^{}f(\Hop(s))$ in $\|\cdot\|_{\il}$ can be established completely analogously. By the telescoping technique we obtain
\begin{align}
\nonumber&\frac{1}{\epsilon}\left\{\prod_{l=1}^{n}(zI-H_{l}(t+\epsilon))^{-k_{l}}-\prod_{l=1}^{n}(zI-H_{l}(t))^{-k_{l}}\right\}\\
\nonumber&=\sum_{j=1}^{n}\bigg\{\Big[\prod_{l=1}^{j-1}(zI-H_{l}(t))^{-k_{l}}\Big]\bigg(\sum_{\substack {1\le p_{0}^{j},p_{1}^{j}\le k_{j}\\p_{0}^{j}+p_{1}^{j}=k_{j}+1}}(zI-H_{j}(t+\epsilon))^{-p_{0}^{j}}V_{j}(zI-H_{j}(t))^{-p_{1}^{j}}\bigg)\\
\nonumber&\hspace{3in}\times\Big[\prod_{l=j+1}^{n}(zI-H_{l}(t+\epsilon))^{-k_{l}}\Big]\bigg\}\\
\label{Kp-Lem1-R6}&=:\sum_{j=1}^{n}\Gamma_{j}^{f}(t+\epsilon)\ \ \mbox{(say)}.
\end{align}
Applying \cite[Lemma 3.6]{DySk14} to \eqref{Kp-Lem1-R2} we obtain
\begin{align}
\nonumber&D_{H_{j}}^{f}(t)\\
\nonumber=&\Big[\prod_{l=1}^{j-1}(zI-H_{l}(t))^{-k_{l}}\Big]\bigg(\sum_{\substack{1\le p_{0}^{j},p_{1}^{j}\le k_{j}\\p_{0}^{j}+p_{1}^{j}=k_{j}+1}}(zI-H_{j}(t))^{-p_{0}^{j}}V_{j}(zI-H_{j}(t))^{-p_{1}^{j}}\bigg)\Big[\prod_{l=j+1}^{n}(zI-H_{l}(t))^{-k_{l}}\Big].
\end{align}
Consider the difference
\begin{align}
\nonumber&\frac{1}{\epsilon}\left\{\prod_{l=1}^{n}(zI-H_{l}(t+\epsilon))^{-k_{l}}-\prod_{l=1}^{n}(zI-H_{l}(t))^{-k_{l}}\right\}-\sum_{j=1}^{n}D_{H_{j}}^{f}(t)\\
\nonumber&\stackrel{\eqref{Kp-Lem1-R6}}{=}\sum_{j=1}^{n}\left(\Gamma_{j}^{f}(t+\epsilon)-D_{H_{j}}^{f}(t)\right).
\end{align}
Now we claim that $\displaystyle\lim_{\epsilon\to0}\left\|\Gamma_{j}^{f}(t+\epsilon)-D_{H_{j}}^{f}(t)\right\|_{\il}=0$ for $1\le j\le n$. We consider only the case $j=1$ (other cases follow analogously). From \eqref{Kp-Lem1-R6} we further derive
\begin{align}
\nonumber&\left\|\Gamma_{1}^{f}(t+\epsilon)-D_{H_{1}}^{f}(t)\right\|_{\il}\\
\nonumber=&\bigg\|\sum_{\substack {1\le p_{0}^{1},p_{1}^{1}\le k_{1} \\ p_{0}^{1}+p_{1}^{1}=k_{1}+1}}(zI-H_{1}(t+\epsilon))^{-p_{0}^{1}}\,V_{1}(zI-H_{1}(t))^{-p_{1}^{1}}\Big[\prod_{l=2}^{n}(zI-H_{l}(t+\epsilon))^{-k_{l}}\Big]-D_{H_{1}}^{f}(t)\,\bigg\|_{\il}\\
\nonumber\le&\epsilon\, k_{1}K\,\|V_{1}\|\left(\sum_{i=1}^{n}k_{i}\cdot\|V_{i}\|_{\il}\right),
\end{align}
for some constant $K$. This completes the proof.
\end{proof}

\begin{rmrk}\label{Kp-Rmrk1}
Despite the fact that Lemma \ref{Kp-Lem1} is true for $f\in\mathfrak{R}$, but it still holds for any function $f\in\RRn$ by using the linearity of the G\^{a}teaux derivative. In that case, for $f\in\RRn$ such that $f=\sum_{l=1}^{m}c_{l}\varphi_{l}$, where $\varphi_{l}\in\mathfrak{R}$, we have
$$D_{H_{j}}^{f}(t)=\sum_{l=1}^{m}c_{l}D_{H_{j}}^{\varphi_{l}}(t)~\text{ and }~D_{H_{i},H_{j}}^{f}(t)=\sum_{l=1}^{m}c_{l}D_{H_{i},H_{j}}^{\varphi_{l}}(t),~\text{ for }~1\le i\le j\le n.$$
\end{rmrk}

\begin{rmrk}\label{Kp-Rmrk2}
In the next section (Section \ref{Sec5}), we focus on maximal dissipative operators. There, we will require Lemma \ref{Kp-Lem1} in the setting of maximal dissipative operators. Specifically, if each self-adjoint operator $H_{j}$ is replaced by a linear operator $L_{j}$ such that $L_{j}$ and $L_{j}+V_{j}$ are maximal dissipative for $j=1,2$, then it is routine to verify that Lemma \ref{Kp-Lem1} holds for every $f\in\mathfrak{R}(\R^{2})^{-}$ (see \eqref{rationallower}).
\end{rmrk}

We shall need the following lemma for our use. The notations used below were introduced in Notation \ref{Notation}.

\begin{lma}\label{Kp-Lem2}
Let $H_{j}$ be self-adjoint operators in $\hil$ and $V_{j}=V_{j}^*\in\ideal$ for $j=1,\ldots,n$. Then, for $f\in\RRn$, the function	
\begin{align}
\label{Kp-Lem2-R1}t\in[0,1]\mapsto\frac{d}{ds}\bigg|_{s=t}^{}f(\Hop(s))-\frac{d}{ds}\bigg|_{s=0}^{}f(\Hop(s))
\end{align}
is uniformly continuous and, hence, Bochner integrable on $[0,1]$ with respect to the ideal norm $\|\cdot\|_{\il}$.
\end{lma}

\begin{proof}
Note that it suffices to prove the lemma for $f\in\mathfrak{R}.$ Equation \eqref{Kp-Lem1-R1} implies that
\begin{align}
\nonumber&\frac{d}{ds}\bigg|_{s=t+\epsilon}^{}f(\Hop(s))-\frac{d}{ds}\bigg|_{s=t}^{}f(\Hop(s))=\sum_{j=1}^{n}\left(D_{H_{j}}^{f}(t+\epsilon)-D_{H_{j}}^{f}(t)\right).
\end{align}
The following estimate is immediate from expression \eqref{Kp-Lem1-R2} and \cite[Lemma 3.6]{DySk14}:
\begin{flalign}
\nonumber&\left\|D_{H_{j}}^{f}(t+\epsilon)-D_{H_{j}}^{f}(t)\right\|_{\il}\le \epsilon\,K\,\,\bigg\{\sum_{\substack{i=1\\i\neq j}}^{n}k_{i}k_{j}\|V_{i}\|_{\ideal}\|V_{j}\|_{\ideal}+2k_{j}^{2}\|V_{j}\|_{\ideal}^{2}\bigg\},
\end{flalign}
for $1\le j\le n$, where $K$ is a constant independent of $\epsilon$. Hence, by the triangle inequality for the ideal norm $\|\cdot\|_{\il}$, the function in \eqref{Kp-Lem2-R1} is uniformly continuous on $[0,1]$.
\end{proof}

Based on the above lemma we obtain the following crucial result. The proof follows the standard approach similar to the proof of \cite[Lemma 4.7]{Sk15}.

\begin{lma}\label{Kp-Lem3}
Let $H_{j}$ be self-adjoint operators in $\hil$ and $V_{j}=V_{j}^*\in\ideal$ for $j=1,\ldots,n$. Then, for every $f\in\RRn$,
\begin{align*}
\Tril\left\{f(\Hop(1))-f(\Hop)-\frac{d}{ds}\bigg|_{s=0}^{}f(\Hop(s))\right\}=\int_{0}^{1}(1-t)\,\Tril\left(\ddds\bigg|_{s=t}^{}f(\Hop(s))\right)dt.
\end{align*}
\end{lma}

To proceed further, we need the following lemma, and its proof follows from \cite[Lemma 3.5]{Sk15}.

\begin{lma}\label{Kp-Lem4}
Let $\Tril$ be a $\|\cdot\|_{\il}$-bounded trace on $\il$ with the Jordan decomposition given by $\Tril=\tau_{1,\il}-\tau_{2,\il}+i\tau_{3,\il}-i\tau_{4,\il}$, where $\tau_{k,\il}$ $(1\le k\le 4)$ are the positive components of $\Tril$. Let $n_{1}, n_{2}, N_{1}, N_{2}\in\N$. Let $\{\delta_{i_{1}}\}_{i_{1}=1}^{N_{1}}$ and $\{\delta_{i_{2}}\}_{i_{2}=1}^{N_{2}}$ be partitions of $\C^{n_{1}}$ and $\C^{n_{2}}$, respectively, and let $E_{1}(\cdot)$ and $E_{2}(\cdot)$ be spectral measures on $\C^{n_{1}}$ and $\C^{n_{2}}$, respectively. Then for $V_{1}, V_{2}\in\ideal,$
$$\sum_{i_{1}=1}^{N_{1}}\sum_{i_{2}=1}^{N_{2}}\Big|\Tril(E_{1}(\delta_{i_{1}})\,V_{1}\,E_{2}(\delta_{i_{2}})\,V_{2}\,E_{1}(\delta_{i_{1}}))\Big|\le\sum_{k=1}^{4}\left(\tau_{k,\il}(|V_{1}|^{2})\right)^{1/2}\,\left(\tau_{k,\il}(|V_{2}|^{2})\right)^{1/2}.$$
\end{lma}

In the second order trace formula, as in the first order case, the divided difference plays a central role. Recall the definition of the divided difference from Subsection \ref{Div-Sec}. Below we outline several useful properties of the divided difference for functions in $\RRn$.

\begin{lma}{\normalfont\cite[Lemmas 4.1, 4.2]{Sk15}}\label{Kp-Lem5}
Let $\phi\in\RRn.$ In that case, the following are true.
\begin{enumerate}[{\normalfont(i)}]
\item For $\mu_{0},\ldots,\mu_{k}\in\R$, we have
\begin{align}
\label{Kp-Lem5-R1}&\sup_{\lambda_{1},\ldots,\lambda_{n},\mu_{0},\ldots,\mu_{k}\in\R}\left|~\phi_{j}^{[k]}(\lambda_1,\ldots,\lambda_{j-1},\mu_{0},\ldots,\mu_{k},\lambda_{j+1},\ldots,\lambda_{n})\right|\le\frac{1}{k!}\left\|\frac{\partial^{k}\phi}{\partial \lambda_{j}^{k}}\right\|_{\infty}.
\end{align}
\item For $\eta_{0},\eta_{1},\mu_{0}, \mu_{1}\in\R$, we have
\begin{align}
\label{Kp-Lem5-R2}&\sup_{\substack{\lambda_{1},\ldots,\lambda_{n},\eta_{0},\eta_{1},\mu_{0},\mu_{1}\in\,\R}}\left|\phi_{i,j}^{[2]}(\lambda_1,\ldots,\lambda_{i-1},\eta_{0},\eta_{1},\lambda_{i+1},\ldots,\lambda_{j-1},\mu_{0},\mu_{1},\lambda_{j+1},\ldots,\lambda_{n})\right|\le\left\|\frac{\partial^{2}\phi}{\partial\lambda_{i}\partial\lambda_{j}}\right\|_{\infty}.
\end{align}
\end{enumerate}
\end{lma}

We now derive representations for the second order divided differences of functions $f\in\mathfrak{R}$.

\begin{lma}\label{Kp-Lem6}
Let $f\in\mathfrak{R}$ be given by
\[f(\lambda_{1},\ldots,\lambda_{n})=(z-\lambda_{1})^{-k_{1}}\cdots(z-\lambda_{n})^{-k_{n}},\]
where $\Im(z)\neq0$ and $k_{1},\ldots,k_{n}\in\N$. Then the following assertions hold.
\begin{enumerate}[{\normalfont(i)}]
\item If $\lambda_{i_{0}}, \lambda_{i_{1}}, \lambda_{j_{0}}, \lambda_{j_{1}}\in\R$, then
\begin{align}
\nonumber& f_{i,j}^{[2]}(\lambda_{1},\ldots,\lambda_{i-1},\lambda_{i_{0}},\lambda_{i_{1}},\lambda_{i+1},\ldots,\lambda_{j-1},\lambda_{j_{0}},\lambda_{j_{1}},\lambda_{j+1},\ldots,\lambda_{n})\\
\nonumber&=\sum_{\substack {1\le p_{0},p_{1}\le k_{i}\\p_{0}+p_{1}=k_{i}+1}}~\sum_{\substack {1\le q_{0},q_{1}\le k_{j} \\ q_{0}+q_{1}=k_{j}+1}}\Big[\prod_{l=1}^{i-1}(z-\lambda_{l})^{-k_{l}}\Big](z-\lambda_{i_{0}})^{-p_{0}}(z-\lambda_{i_{1}})^{-p_{1}}\Big[\prod_{l=i+1}^{j-1}(z-\lambda_{l})^{-k_{l}}\Big]\\
\label{Kp-Lem6-R1}&\hspace{2in}\times (z-\lambda_{j_{0}})^{-q_{0}}(z-\lambda_{j_{1}})^{-q_{1}}\Big[\prod_{l=j+1}^{n}(z-\lambda_{l})^{-k_{l}}\Big].
\end{align}
\item If $\lambda_{i_{0}}, \lambda_{i_{1}}, \lambda_{i_{2}}\in\R$, then
\begin{align}
\nonumber&f_{i}^{[2]}(\lambda_{1},\ldots,\lambda_{i-1},\lambda_{i_{0}},\lambda_{i_{1}},\lambda_{i_{2}},\lambda_{i+1},\ldots,\lambda_{n})\\
\label{Kp-Lem6-R2}&=\sum_{\substack {1\le p_{0},p_{1},p_{2}\le k_{i} \\ p_{0}+p_{1}+p_{2}=k_{i}+2}}\Big[\prod_{l=1}^{i-1}(z-\lambda_{l})^{-k_{l}}\Big](z-\lambda_{i_{0}})^{-p_{0}}(z-\lambda_{i_{1}})^{-p_{1}}(z-\lambda_{i_{2}})^{-p_{2}}\Big[\prod_{l=i+1}^{n}(z-\lambda_{l})^{-k_{l}}\Big].
\end{align}
\end{enumerate}
\end{lma}

\begin{proof}
The proof follows trivially from \eqref{Div-Diff-1}, \eqref{Div-Diff-2} and \cite[Lemma 6.1]{PoSuUsZa15}.
\end{proof}

Similar to the estimate \eqref{Kr-Thm1-R2} for the first order trace formula, we also need to establish the following crucial estimate. The proof of the theorem below is based on ideas from the proof of \cite[Theorem 4.6]{Sk15}, but with suitable modifications to accommodate unbounded operators. Before we begin, recall  Hypothesis \ref{Assumption} and the expression for $D_{H_{i},H_{j}}^{f}(t)$ from \eqref{Kp-Lem1-R4} and \eqref{Kp-Lem1-R5}.

\begin{thm}\label{Kp-Thm1}
Assume Notation \ref{Notation}. Let $\Tril$ be a $\|\cdot\|_{\il}$-bounded trace on $\il$ with the Jordan decomposition given by $\Tril=\tau_{1,\il}-\tau_{2,\il}+i\tau_{3,\il}-i\tau_{4,\il}$, where $\tau_{k,\il}$ $(1\le k\le 4)$ are the positive components of $\Tril$. Let $H_{j}$ be self-adjoint operators in $\hil$ and $V_{j}=V_{j}^*\in\ideal$ for $j=1,\ldots,n$. Suppose that there exists $t\in[0,1]$ such that $\Hop(t)\in\comn$. Then, for every $f\in\RRn$, 
\begin{align}
\label{Kp-Thm1-R1}&\left|\Tril\left(D_{H_{i},H_{j}}^{f}(t)\right)\right|\le\sum_{k=1}^{4}\left(\tau_{k,\il}(|V_{i}|^{2})\right)^{1/2}\,\left(\tau_{k,\il}(|V_{j}|^{2})\right)^{1/2}\,\,\left\|\frac{\partial^{2} f}{\partial \lambda_{i}\,\partial \lambda_{j}}\right\|_{L^{\infty}(\Rn)},
\end{align}
with $1\le i\le j\le n$.
\end{thm}

\begin{proof}
Although the theorem is stated for $f\in\RRn$, we shall prove it only for $f\in\mathfrak{R}$ (cf. \eqref{Rational-without-span}) for notational simplicity. Let
\[f(\lambda_{1},\ldots,\lambda_{n})=(z-\lambda_{1})^{-k_{1}}\cdots(z-\lambda_{n})^{-k_{n}},\ \ \mbox{$\Im(z)\neq0$ and $k_{1},\ldots,k_{n}\in\N$}.\]  
From \eqref{Kp-Thm1-R1} and the Jordan decomposition of $\Tril$, it is straightforward to see that it suffices to prove \eqref{Kp-Thm1-R1} only for a positive trace $\Tril$.
\begin{case}\label{Kp-Thm1-Case1}
Consider the case $i<j$. Then from \eqref{Kp-Lem1-R4} and \cite[Lemma 3.6]{DySk14} we derive
\begin{align}
\nonumber&\Tril\left(D_{H_{i},H_{j}}^{f}(t)\right)\\
\nonumber=&\sum_{\substack {1\le p_{0},p_{1}\le k_{i}\\p_{0}+p_{1}=k_{i}+1}}~\sum_{\substack {1\le q_{0},q_{1}\le k_{j} \\ q_{0}+q_{1}=k_{j}+1}}\Tril\bigg\{\Big[\prod_{l=1}^{i-1}(zI-H_{l}(t))^{-k_{l}}\Big](zI-H_{i}(t))^{-p_{0}}V_{i}(zI-H_{i}(t))^{-p_{1}}\\
\nonumber&\times \Big[\prod_{l=i+1}^{j-1}(zI-H_{l}(t))^{-k_{l}}\Big](zI-H_{j}(t))^{-q_{0}}V_{j}(zI-H_{j}(t))^{-q_{1}}\Big[\prod_{l=j+1}^{n}(zI-H_{l}(t))^{-k_{l}}\Big]\bigg\}.
\end{align}
Now the cyclicity of the trace $\Tril$ and the pairwise commutativity of the tuple $\Hop(t)$ gives
\begin{align}
\nonumber&\Tril\left(D_{H_{i},H_{j}}^{f}(t)\right)\\
\nonumber=&\sum_{\substack {1\le p_{0},p_{1}\le k_{i} \\ p_{0}+p_{1}=k_{i}+1}}~\sum_{\substack {1\le q_{0},q_{1}\le k_{j}\\q_{0}+q_{1}=k_{j}+1}}\Tril\bigg\{\Big[\prod_{l=1}^{i-1}(zI-H_{l}(t))^{-k_{l}}\Big](zI-H_{i}(t))^{-p_{0}}(zI-H_{j}(t))^{-q_{1}}\\
\label{Kp-Thm1-R4}\times&\Big[\prod_{l=j+1}^{n}(zI-H_{l}(t))^{-k_{l}}\Big]V_{i}\,(zI-H_{i}(t))^{-p_{1}}\Big[\prod_{l=i+1}^{j-1}(zI-H_{l}(t))^{-k_{l}}\Big](zI-H_{j}(t))^{-q_{0}}V_{j}\bigg\}.
\end{align}
By the spectral theorem, for every $1\le l(\neq i,j)\le n$, there is a sequence of Borel partitions $\left(\delta_{m,l,\alpha_{l}}\right)_{1\le \alpha_{l}\le m}$ of $\R$ and a sequence of real numbers $\left(\lambda_{m,l,\alpha_{l}}\right)_{1\le \alpha_{l}\le m}$ such that
$$\sum_{1\le\alpha_{l}\le m}(z-\lambda_{m,l,\alpha_{l}})^{-k_{l}}E_{H_{l}(t)}(\delta_{m,l,\alpha_{l}})$$
converges in operator norm to $(zI-H_{l}(t))^{-k_{l}}$ as $m\to\infty$. And for $l\in\{i,j\}$, there is a sequence of Borel partitions $\left(\delta_{m,l,\alpha_{l}}\right)_{1\le \alpha_{l}\le m}$ of $\R$, a sequence of real numbers $\left(\lambda_{m,l,\alpha_{l}}\right)_{1\le \alpha_{l}\le m}$ such that for any $k\in\{1,\ldots,k_{l}+1\},$ the quantity 
\begin{align}\label{Kp-Thm1-R2}
\sum_{1\le\alpha_{l}\le m}(z-\lambda_{m,l,\alpha_{l}})^{-k}E_{H_{l}(t)}(\delta_{m,l,\alpha_{l}})
\end{align}
converges in operator norm to $(zI-H_{l}(t))^{-k}$ as $m\to\infty$. Here, $E_{H_{l}(t)}(\cdot)$ denotes the spectral measure associated with the self-adjoint operator $H_{l}(t)$ for $1\le l\le n$.
			
Since the tuple $(H_{1}(t),\ldots,H_{n}(t))$ is commuting, $$E_{t}(\delta_{1}\times\cdots\times\delta_{n}):=E_{H_{1}(t)}(\delta_{1})\times\cdots\times E_{H_{n}(t)}(\delta_{n})$$ 
defines a spectral measure. Consequently,
\begin{align}
\label{Kp-Thm1-R3}&\sum_{1\le\alpha_{1},\ldots,\alpha_{n}\le m}\Big[\prod_{l=1}^{n}(z-\lambda_{m,l,\alpha_{l}})^{-k_{l}}\Big]E_{t}(\delta_{m,1,\alpha_{1}}\times\cdots\times\delta_{m,n,\alpha_{n}})
\end{align}
converges in operator norm to $\left((zI-H_{1}(t))^{-k_{1}}\cdots(zI-H_{n}(t))^{-k_{n}}\right)$ as $m\to\infty$. 

Using \eqref{Kp-Thm1-R4}, \eqref{Kp-Thm1-R2} and \eqref{Kp-Thm1-R3} we now obtain
\begin{align}
\nonumber\Tril&\left(D_{H_{i},H_{j}}^{f}(t)\right)=\lim_{m\to\infty}\sum_{1\le\alpha_{1},\ldots,\alpha_{n},\alpha'_{i},\alpha'_{j}\,\le m}\,\sum_{\substack {1\le p_{0},p_{1}\le k_{i}\\p_{0}+p_{1}=k_{i}+1}}\,\sum_{\substack {1\le q_{0},q_{1}\le k_{j}\\q_{0}+q_{1}=k_{j}+1}}\bigg\{\Big[\prod_{l=1}^{i-1}(z-\lambda_{m,l,\alpha_{l}})^{-k_{l}}\Big]\\
\nonumber&\times (z-\lambda_{m,i,\alpha_{i}})^{-p_{0}}(z-\lambda_{m,i,\alpha'_{i}})^{-p_{1}}\Big[\prod_{l=i+1}^{j-1}(z-\lambda_{m,l,\alpha_{l}})^{-k_{l}}\Big](z-\lambda_{m,j,\alpha'_{j}})^{-q_{0}}(z-\lambda_{m,j,\alpha_{j}})^{-q_{1}}\\
\nonumber&\times \Big[\prod_{l=j+1}^{n}(z-\lambda_{m,l,\alpha_{l}})^{-k_{l}}\Big]\bigg\}\Tril\bigg\{E_{1,t}\left(\delta_{m,1,\alpha_{1}}\times\cdots\times\delta_{m,i,\alpha_{i}}\times\delta_{m,j,\alpha_{j}}\times\cdots\times\delta_{m,n,\alpha_{n}}\right)V_{i}\\
\label{Kp-Thm1-R5}&\times E_{2,t}\left(\delta_{m,i,\alpha'_{i}}\times\delta_{m,i+1,\alpha_{i+1}}\times\cdots\times\delta_{m,j-1,\alpha_{j-1}}\times\delta_{m,j,\alpha'_{j}}\right)V_{j}\bigg\},
\end{align}
where $E_{1,t}:=E_{H_{1}(t)}\times\cdots\times E_{H_{i}(t)}\times E_{H_{j}(t)}\times\cdots\times E_{H_{n}(t)}$ and $E_{2,t}:=E_{H_{i}(t)}\times E_{H_{i+1}(t)}\times\cdots\times E_{H_{j-1}(t)}\times E_{H_{j}(t)}.$ Applying \eqref{Kp-Lem6-R1} to equation \eqref{Kp-Thm1-R5} yields
\begin{align}
\nonumber&\Tril\left(D_{H_{i},H_{j}}^{f}(t)\right)=\lim_{m\to\infty}\sum_{1\le\alpha_{1},\ldots,\alpha_{n},\alpha'_{i},\alpha'_{j}\,\le m}\\
\nonumber&f_{i,j}^{[2]}\left(\lambda_{m,1,\alpha_{1}},\ldots,\lambda_{m,i-1,\alpha_{i-1}},\lambda_{m,i,\alpha_{i}},\lambda_{m,i,\alpha'_{i}},\ldots,\lambda_{m,j,\alpha'_{j}},\lambda_{m,j,\alpha_{j}},\lambda_{m,j+1,\alpha_{j+1}},\ldots,\lambda_{m,n,\alpha_{n}}\right)\\
\nonumber&\times\Tril\left\{E_{1,t}(\delta_{m,1,\alpha_{1}}\times\cdots\times\delta_{m,i,\alpha_{i}}\times\delta_{m,j,\alpha_{j}}\times\cdots\times\delta_{m,n,\alpha_{n}})V_{i}\right.\\
\label{Kp-Thm1-R6}&\hspace{4cm}\left. E_{2,t}(\delta_{m,i,\alpha'_{i}}\times\delta_{m,i+1,\alpha_{i+1}}\times\cdots\times\delta_{m,j-1,\alpha_{j-1}}\times\delta_{m,j,\alpha'_{j}})V_{j}\right\}.
\end{align}
Therefore, using \eqref{Kp-Lem5-R2} and Lemma \ref{Kp-Lem4}, the above equation \eqref{Kp-Thm1-R6} gives the following estimate:
$$\left|\Tril\left(D_{H_{i},H_{j}}^{f}(t)\right)\right|\le\left(\Tril(|V_{i}|^{2})\right)^{1/2}\,\left(\Tril(|V_{j}|^{2})\right)^{1/2}\cdot\left\|\frac{\partial^{2} f}{\partial \lambda_{i}\,\partial \lambda_{j}}\right\|_{L^{\infty}(\Rn)}.$$
This completes the proof for the case $i<j$.
\end{case}
	
\begin{case}\label{Kp-Thm1-Case2}
Next, assume $i=j$. Using \eqref{Kp-Lem1-R5} and \cite[Lemma 5.4 (i)]{DySk14} we obtain
\begin{align}
\nonumber&\Tril\left(D_{H_{i},H_{i}}^{f}(t)\right)=2\sum_{\substack {1\le p_{0},p_{1},p_{2}\le k_{i} \\ p_{0}+p_{1}+p_{2}=k_{i}+2}}\Tril\bigg\{\Big[\prod_{l=1}^{i-1}(zI-H_{l}(t))^{-k_{l}}\Big](zI-H_{i}(t))^{-p_{0}}V_{i}(zI-H_{i}(t))^{-p_{1}}V_{i}\\
\nonumber&\hspace{2.5in}\times (zI-H_{i}(t))^{-p_{2}}\Big[\prod_{l=i+1}^{n}(zI-H_{l}(t))^{-k_{l}}\Big]\bigg\}.
\end{align}
Again, the cyclicity of the trace $\Tril$ and the pairwise commutativity of the tuple $\Hop(t)$ imply
\begin{align}
\nonumber&\Tril\left(D_{H_{i},H_{i}}^{f}(t)\right)=2\sum_{\substack{1\le p_{0},p_{1},p_{2}\le k_{i} \\ p_{0}+p_{1}+p_{2}=k_{i}+2}}\Tril\bigg\{\Big[\prod_{l=1}^{i-1}(zI-H_{l}(t))^{-k_{l}}\Big](zI-H_{i}(t))^{-(p_{0}+p_{2})}\\
\nonumber&\hspace{2.4in}\times \Big[\prod_{l=i+1}^{n}(zI-H_{l}(t))^{-k_{l}}\Big]V_{i}\,(zI-H_{i}(t))^{-p_{1}}\,V_{i}\bigg\}.
\end{align}
Now, by the spectral theorem we have
\begin{align}
\nonumber&\Tril\left(D_{H_{i},H_{i}}^{f}(t)\right)=2\lim_{m\to\infty}\sum_{1\le\alpha_{1},\ldots,\alpha_{n},\alpha'_{i}\le m}\sum_{\substack {1\le p_{0},p_{1},p_{2}\le k_{i}\\p_{0}+p_{1}+p_{2}=k_{i}+2}}\hspace*{-0.2cm}\bigg\{\Big[\prod_{l=1}^{i-1}(z-\lambda_{m,l,\alpha_{l}})^{-k_{l}}\Big](z-\lambda_{m,i,\alpha_{i}})^{-(p_{0}+p_{2})}\\
\label{Kp-Thm1-R9}&\times (z-\lambda_{m,i,\alpha'_{i}})^{-p_{1}}\Big[\prod_{l=i+1}^{n}(z-\lambda_{m,l,\alpha_{l}})^{-k_{l}}\Big]\bigg\}\Tril\bigg\{E_{1,t}(\delta_{m,1,\alpha_{1}}\times\cdots\times\delta_{m,n,\alpha_{n}})\,V_{i}\,E_{2,t}(\delta_{m,i,\alpha'_{i}})\,V_{i}\bigg\},
\end{align}
where $E_{1,t}:=E_{H_{1}(t)}\times\cdots\times E_{H_{n}(t)}$ and $E_{2,t}:=E_{H_{i}(t)}$. Applying \eqref{Kp-Lem6-R2} to \eqref{Kp-Thm1-R9}, we further deduce
\begin{align}
\nonumber&\Tril\left(D_{H_{i},H_{i}}^{f}(t)\right)=\\
\nonumber&2\lim_{m\to\infty}\sum_{1\le\alpha_{1},\ldots,\alpha_{n},\alpha'_{i}\le m}f_{i}^{[2]}\left(\lambda_{m,1,\alpha_{1}},\ldots,\lambda_{m,i-1,\alpha_{i-1}},\lambda_{m,i,\alpha_{i}},\lambda_{m,i,\alpha'_{i}},\lambda_{m,i,\alpha_{i}},\lambda_{m,i+1,\alpha_{i+1}},\ldots,\lambda_{m,n,\alpha_{n}}\right)\\
\label{Kp-Thm1-R10}&\hspace*{4cm}\times\Tril\left(E_{1,t}\left(\delta_{m,1,\alpha_{1}}\times\cdots\times\delta_{m,n,\alpha_{n}}\right)\,V_{i}\,E_{2,t}(\delta_{m,i,\alpha'_{i}})\,V_{i}\right),
\end{align}
and therefore, applying \eqref{Kp-Lem5-R1} and Lemma \ref{Kp-Lem4} to \eqref{Kp-Thm1-R10}, we establish \eqref{Kp-Thm1-R1} for the case $i=j$. This completes the proof.
\end{case}
\end{proof}

We now establish the Koplienko trace formula for an $n$-tuple of commuting self-adjoint operators.

\begin{thm}\label{Kp-Thm2}
Assume Notation \ref{Notation}. Let $\Tril$ be a $\|\cdot\|_{\il}$-bounded trace on $\il$ with the Jordan decomposition given by $\Tril=\tau_{1,\il}-\tau_{2,\il}+i\tau_{3,\il}-i\tau_{4,\il}$, where $\tau_{k,\il}$ $(1\le k\le 4)$ are the positive components of $\Tril$. Let $H_{j}$ be self-adjoint operators in $\hil$ and $V_{j}=V_{j}^*\in\ideal$ for $j=1,\ldots,n$. If $\Hop(t)\in\comn$ for all $t\in[0,1]$, then there exist finite measures $\nu_{ij}$, $1\le i\le j\le n$, on $\Rn$ such that
\begin{align}
\label{Kp-Thm2-R1}&\|\nu_{ij}\|\le\frac{1}{2}\cdot\sum_{k=1}^{4}\left(\tau_{k,\il}(|V_{i}|^{2})\right)^{1/2}\,\left(\tau_{k,\il}(|V_{j}|^{2})\right)^{1/2},
\end{align}
and
\begin{align}
\label{Kp-Thm2-R2}\nonumber&\Tril\left\{f(\Hop(1))-f(\Hop)-\frac{d}{ds}\bigg|_{s=0}^{}f(\Hop(s))\right\}\\
\nonumber&\hspace*{1in}=2\sum_{1\le i<j\le n}\int_{\Rn}\frac{\partial^{2} f}{\partial \lambda_{i}\,\partial \lambda_{j}}(\lambda_{1},\ldots,\lambda_{n})\,d\nu_{ij}(\lambda_{1},\ldots,\lambda_{n})\\
&\hspace*{1in}\quad+\sum_{1\le j\le n}\int_{\Rn}\frac{\partial^{2} f}{\partial \lambda_{j}^{2}} (\lambda_{1},\ldots,\lambda_{n})\,d\nu_{jj}(\lambda_{1},\ldots,\lambda_{n}),
\end{align}
for every $f\in\RRn$.
\end{thm}

\begin{proof}
By Lemma \ref{Kp-Lem1} and Lemma \ref{Kp-Lem3}, we obtain
\begin{align*}
&\Tril\left\{f(\Hop(1))-f(\Hop)-\frac{d}{ds}\bigg|_{s=0}^{}f(\Hop(s))\right\}\\
&=2\sum_{1\le i<j\le n}\int_{0}^{1}(1-t)\Tril\left(D_{H_{i},H_{j}}^{f}(t)\right)dt+\sum_{1\le j\le n}\int_{0}^{1}(1-t)\Tril\left(D_{H_{j},H_{j}}^{f}(t)\right)dt.
\end{align*}
Then the collection of linear functionals $\Psi_{ij}:~\textup{span}\left\{\frac{\partial^{2}f}{\partial\lambda_{i}\,\partial\lambda_{j}}~:~f\in\RRn\right\}\to\C$, defined by
\begin{align*}
&\Psi_{ij}\left(\frac{\partial^{2} f}{\partial\lambda_{i}\,\partial \lambda_{j}}\right)=\int_{0}^{1}(1-t)\Tril\left(D_{H_{i},H_{j}}^{f}(t)\right)dt,\quad\text{for }~1\le i\le j\le n,
\end{align*}
is bounded by Theorem \ref{Kp-Thm1}. Moreover, for each $1\le i\le j\le n$,
\begin{flalign*}
&\left|\Psi_{ij}\left(\frac{\partial^{2} f}{\partial \lambda_{i}\,\partial \lambda_{j}}\right)\right|\le \frac{1}{2}\cdot\sum_{k=1}^{4}\left(\tau_{k,\il}(|V_{i}|^{2})\right)^{1/2}\,\left(\tau_{k,\il}(|V_{j}|^{2})\right)^{1/2}\left\|\frac{\partial^{2} f}{\partial \lambda_{i}\,\partial \lambda_{j}}\right\|_{L^{\infty}(\Rn)}.
\end{flalign*}
Therefore, by the Hahn-Banach theorem and the Riesz representation theorem for elements of $(C_{0}(\Rn))^*$, there exist finite measures
$\nu_{ij}$ (for $1\le i\le j\le n$) on $\Rn$ satisfying \eqref{Kp-Thm2-R1} and \eqref{Kp-Thm2-R2}. This completes the proof.
\end{proof}

\begin{thm}\label{Kp-Thm3}
Assume Notation \ref{Notation}. Let $\mathcal{M}_\psi$ be a Lorentz ideal, where the function $\psi$ satisfies \eqref{psi-growth-condition} for some $0<\varepsilon<1/3$. Let $\tau_{\psi}$ be a bounded singular trace on $\mathcal{M}_\psi$. Let $H_{j}$ be self-adjoint operators in $\hil$ and $V_{j}=V_{j}^*\in\mathcal{M}_\psi^{1/2}$ for $j=1,\ldots,n$. If $\Hop\in\comn$, then for every $f\in\RRn$,
\begin{align}
\label{Kp-Thm3-R1}&\tau_\psi\left\{f(\Hop(1))-f(\Hop)-\frac{d}{ds}\bigg|_{s=0}^{}f(\Hop(s))\right\}=\frac{1}{2}\tau_\psi\left(\frac{d^{2}}{ds^{2}}\bigg|_{s=0}^{}f(\Hop(s))\right),
\end{align}
and there exist finite measures $\nu_{ij}$, $1\le i\le j\le n$, on $\Rn$ such that \eqref{Kp-Thm2-R2} holds.
\end{thm}

\begin{proof}
The proof is analogous to the proof of Theorem \ref{Kr-Thm3}. Firstly one establishes
\begin{align}
\nonumber&\left(f(\Hop(1))-f(\Hop)-\frac{d}{ds}\bigg|_{s=0}^{}f(\Hop(s))-\frac{1}{2}\frac{d^{2}}{ds^{2}}\bigg|_{s=0}^{}f(\Hop(s))\right)\in\mathcal{M}_\psi^{3/2}.
\end{align}
Then, using \eqref{Dixmier-Property}, we obtain the identity \eqref{Kp-Thm3-R1}. Furthermore, from \eqref{Kp-Lem1-R3} we derive
\begin{align*}
\tau_\psi&\left\{f(\Hop(1))-f(\Hop)-\frac{d}{ds}\bigg|_{s=0}^{}f(\Hop(s))\right\}\\
&=\sum_{1\le i<j\le n}\tau_\psi\left(D_{H_{i},H_{j}}^{f}(0)\right)+\frac{1}{2}\sum_{1\le j\le n}\tau_\psi\left(D_{H_{j},H_{j}}^{f}(0)\right).
\end{align*}
Finally, by mimicking the proof of Theorem \ref{Kp-Thm2}, we complete the proof.
\end{proof}

\section{Trace formulas for maximal dissipative operators}\label{Sec5}
In this section, we derive the Krein and Koplienko trace formulas for a pair of maximal dissipative operators using the main results of Sections \ref{Sec3} and \ref{Sec4}. We begin with the following definition.

\begin{dfn}\label{D-Def}
A densely defined linear operator $L$ (possibly unbounded) in $\hil$ with domain $\text{Dom}(L)$ is called {\it dissipative} if $\Im\la Lh, h\ra\ge0$ for every $h\in\text{Dom}(L)$. It is called {\it maximal dissipative} if it does not have a proper dissipative extension.
\end{dfn}

Every dissipative operator has a maximal dissipative extension. Every maximal dissipative operator $L$ is necessarily closed and its spectrum $\sigma(L)$ is contained in the closed upper half-plane $\clos \C_{+}$. The Cayley transform of a maximal dissipative operator $L$ is defined by 
\begin{align*}
T\overset{\mathrm{def}}{=}(L-iI)(L+iI)^{-1}.
\end{align*}
Then $T$ is a contraction on $\hil$ and, importantly, $1$ is not an eigenvalue of $T$. The basic theory of dissipative operators can be found in \cite[Chap. IV, Sect. 4]{SzNF70}.

Let $L$ and $M$ be two maximal dissipative operators in $\hil$. We say that $L$ and $M$ commute if they commute in the resolvent sense, that is, if their resolvents $(L+iI)^{-1}$ and $(M+iI)^{-1}$ commute. It is easy to verify that $(L+iI)^{-1}$ and $(M+iI)^{-1}$ commute if and only if the Cayley transforms $(L-iI)(L+iI)^{-1}$ and $(M-iI)(M+iI)^{-1}$ of $L$ and $M$ commute (see, e.g., \cite[Theorem 3.5]{AlPe22}). 

Let us introduce the notion of resolvent self-adjoint dilation for a commuting tuple of maximal dissipative operators.

\begin{dfn}
Let $(L_{1},\ldots,L_{n})$ be an $n$-tuple of commuting maximal dissipative operators in $\hil$, and let $(X_{1},\ldots,X_{n})$ be an $n$-tuple of commuting self-adjoint operators in a Hilbert space $\mathcal{K}\supset\hil$. The operator tuple $(X_{1},\ldots,X_{n})$ is called a {\it resolvent self-adjoint dilation} of $(L_{1},\ldots,L_{n})$ if
\begin{align}\label{Dilation}
(zI-L_{i_{1}})^{-1}\cdots (zI-L_{i_{k}})^{-1}=P_{\hil}(zI-X_{i_{1}})^{-1}\cdots (zI-X_{i_{k}})^{-1}\big|_{\hil},\ \ \Im(z)<0,
\end{align}
for all $k\ge1$ and all choices of indices $1\le i_{1},\ldots,i_{k}\le n$. Here $P_{\hil}$ denotes the orthogonal projection from $\mathcal{K}$ onto $\hil$.
\end{dfn}

\begin{rmrk}
Note that for $n>2$, an $n$-tuple of commutative maximal dissipative operators (contractions) does
not always admit an $n$-tuple of commutative resolvent self-adjoint (unitary) dilations \cite[p. 23]{SzNF70}, whereas for $n=2$ such dilations always exist (see \cite[Chap. I, Sec. 6, Thm. 6.4]{SzNF70}). Henceforth, in this section we fix $n=2$.
\end{rmrk}

We introduce several notations. Let $\rcomn$ denote the class of all pairs $(L_{1},L_{2})$ of maximal dissipative operators in $\hil$, with commuting resolvents. Let $\mathfrak{R}^{-}$ denote the set of functions
\begin{align}\label{Rational-without-span-}
\left\{ (\lambda_{1},\lambda_{2})\in\R^{2}\mapsto(z-\lambda_{1})^{-k_{1}}(z-\lambda_{2})^{-k_{2}}:\,\,\Im(z)<0,\,k_{1},k_{2}\in\N\right\}.
\end{align}
Note that the class $\mathfrak{R}(\mathbb R^2)^{-}$, defined in \eqref{rationallower}, coincide with the linear span of $\mathfrak{R}^{-}$.

For $f\in\mathfrak{R}^{-}$, the corresponding operator function associated with a pair of maximal dissipative operators $(L_{1},L_{2})$ (not necessarily commutative) is defined by
$$f(L_{1},L_{2})=(zI-L_{1})^{-k_{1}}(zI-L_{2})^{-k_{2}}.$$
We adhere to the notational conventions established in Notation \ref{Notation}. Moreover, for a given maximal dissipative operator $L$ and a bounded perturbation operator $V$, we set
\[L(t):=L+tV,\ \ t\in[0,1].\]
Note that $L(t)$ is maximal dissipative for each $t\in[0,1]$ whenever both $L$ and $L+V$ are maximal dissipative operators (cf. \cite[Lemma 4.1]{AlPe11}).

\subsection{Krein trace formula for maximal dissipative operators}
In the preceding sections, we derived trace formulas for self-adjoint operators with perturbations in arbitrary symmetrically normed ideals under Hypothesis \ref{Assumption}. These results rely on the assumption that the perturbed tuple $\Hop(t)$ is commutative for all $t\in[0,1]$. On a technical level, extending these results to maximal dissipative operators requires assuming that the perturbed pair $(L_{1}(t),L_{2}(t))\in\rcomn$ for all $t\in[0,1]$. Since no natural or representative examples supporting this assumption are known to authors, we find it very restrictive.

However, if we limit our considerations to perturbations belonging to the Lorentz ideals, then it is sufficient to assume $(L_{1}(t),L_{2}(t))\in\rcomn$ only at the initial point $t=0$. Therefore we prove our main results in this section for perturbations in the Lorentz ideals. Note, however, that using a similar method, the results can be established for any symmetrically normed ideal under more restrictive assumptions (see Remark \ref{D-Rmrk2}).

\smallskip

We now proceed to establish the Krein trace formula for a pair of maximal dissipative operators $(L_{1},L_{2})$ with perturbations belonging to the Lorentz ideal $\mathcal{M}_{\psi}$, where $\mathcal{M}_{\psi}$ is endowed with a singular trace $\tau_{\psi}$ bounded with respect to $\|\cdot\|_{\mathcal{M}_{\psi}}$.

\begin{thm}\label{D-Thm1}
Let $\psi:(0,\infty)\to(0,\infty)$ be a concave function such that $\lim_{t\to\infty} \psi(t)=\infty$ and the condition \eqref{psi-growth-condition} is satisfied  for some $0<\varepsilon<1/2$. Let $\mathcal{M}_\psi$ be the corresponding Lorentz ideal and let $\tau_{\psi}$ be a bounded singular trace on $\mathcal{M}_\psi$. Let $V_{1},V_{2}\in\mathcal{M}_\psi$, and let $L_{1},L_{2}$ be two maximal dissipative operators in $\hil$ such that $L_{1}+V_{1}$ and $L_{2}+V_{2}$ are also maximal dissipative. If $(L_{1},L_{2})\in\rcomn$, then there exist finite measures $\mu_{1},\mu_{2}$ on $\R^{2}$  such that 
\begin{align}
\label{D-Thm1-R2}&\tau_{\psi}\left\{f(L_{1}+V_{1},L_{2}+V_{2})-f(L_{1},L_{2})\right\}=\sum_{j=1}^{2}\int_{\R^{2}}\frac{\partial f}{\partial \lambda_{j}}(\lambda_{1},\lambda_{2})\,d\mu_{j}(\lambda_{1},\lambda_{2}),
\end{align}
for every $f\in\mathfrak{R}(\R^{2})^-$.
\end{thm}

\begin{proof}
In order to prove \eqref{D-Thm1-R2} for all such $f\in\mathfrak{R}(\R^{2})^-$, it will be enough to consider $f\in\mathfrak{R}^-$ (see \eqref{rationallower} and \eqref{Rational-without-span-} resp.). Let
\[f(\lambda_{1},\lambda_{2})=(z-\lambda_{1})^{-k_{1}}(z-\lambda_{2})^{-k_{2}},\ \ \mbox{$\Im(z)<0$ and $k_{1},k_{2}\in\N$}.\]
Similarly to the proof of Theorem \ref{Kr-Thm3}, we establish the following identity:
\begin{align*}
\tau_{\psi}\left\{f(L_{1}+V_{1},L_{2}+V_{2})-f(L_{1},L_{2})\right\}=\tau_{\psi}\left(\frac{d}{dt}\bigg|_{t=0}f(L_{1}(t),L_{2}(t))\right).
\end{align*}
From Lemma \ref{Kp-Lem1} (see Remark \ref{Kp-Rmrk2}) we further obtain 
\begin{align}\label{D-Thm1-R3}
\tau_{\psi}\left\{f(L_{1}+V_{1},L_{2}+V_{2})-f(L_{1},L_{2})\right\}=\sum_{j=1}^{2}\tau_{\psi}\left(D^{f}_{L_{j}}(0)\right),
\end{align}
where
\begin{align}\label{D-Thm1-R4}
D^{f}_{L_{j}}(0)=\sum_{\substack{1\le p_{0}^{j},p_{1}^{j}\le k_{j}\\p_{0}^{j}+p_{1}^{j}=k_{j}+1}}\Big[\prod_{l=1}^{j-1}(zI-L_{l})^{-k_{l}}\Big](zI-L_{j})^{-p_{0}^{j}}V_{j}(zI-L_{j})^{-p_{1}^{j}}\Big[\prod_{l=j+1}^{2}(zI-L_{l})^{-k_{l}}\Big].
\end{align}
Since $(L_{1},L_{2})\in\rcomn$, from \eqref{D-Thm1-R3} and \eqref{D-Thm1-R4} we obtain
\begin{align}\label{D-Thm1-R5}
\tau_{\psi}\left\{f(L_{1}+V_{1},L_{2}+V_{2})-f(L_{1},L_{2})\right\}=\sum_{j=1}^{2}\tau_{\psi}\left(V_{j}\frac{\partial f}{\partial\lambda_{j}}(L_{1},L_{2})\right).
\end{align}

Let $(X_{1},X_{2})\in\text{Com}_{2}$ be the resolvent self-adjoint dilation of the tuple $(L_{1},L_{2})$ on a Hilbert space $\mathcal{K}\supset\hil$ such that \eqref{Dilation} holds. Then
\begin{align}\label{D-Thm1-R6}
\frac{\partial f}{\partial \lambda_{j}}(L_{1},L_{2})=P_{\hil}\frac{\partial f}{\partial \lambda_{j}}(X_{1},X_{2})\bigg|_{\hil},\ \ j=1,2,
\end{align}
where $P_{\hil}$ is a projection from $\mathcal{K}$ onto $\hil$.

Let $\tau_{\psi}=\tau_{1,\psi}-\tau_{2,\psi}+i\tau_{3,\psi}-i\tau_{4,\psi}$ be the Jordan decomposition of $\tau_{\psi}:\mathcal{M}_{\psi}\to\C$, where $\tau_{i,\psi}$ for $1\le i\le 4$ are positive bounded traces (see, e.g., \cite[Theorem 4.2.2]{LoSuZabook1ed}). Then, by \cite[Proposition 2.3]{Sk15}, there exists an ideal $\mathcal{I}\subseteq\mathcal{B}(\mathcal{K})$ with the ideal norm $\|\cdot\|_{\mathcal{I}}$ and bounded positive traces $\tau_{i,\mathcal{I}}$ for $1\le i\le 4$ on $\mathcal{I}$, whose restriction to $\mathcal{M}_{\psi}$ coincide with $\tau_{i,\psi}$ for $1\le i\le 4$. Define 
\begin{align*}
\tau_{\mathcal{I}}=\tau_{1,\mathcal{I}}-\tau_{2,\mathcal{I}}+i\tau_{3,\mathcal{I}}-i\tau_{4,\mathcal{I}}.
\end{align*}
Then $\tau_{\mathcal{I}}$ is also a bounded trace on $\mathcal{I}$ whose restriction to $\mathcal{M}_{\psi}$ coincides with $\tau_{\psi}$.
	
By the spectral theorem and Lemma \ref{Kr-Lem4}, we have
\begin{align}
\nonumber\left|\tau_{\mathcal{I}}\left(\frac{\partial f}{\partial \lambda_{j}}(X_{1},X_{2})V_{j}\right)\right|&\le(\tau_{1,\mathcal{I}}+\tau_{2,\mathcal{I}}+\tau_{3,\mathcal{I}}+\tau_{4,\mathcal{I}})\left(|\Re(V_{j})|+|\Im(V_{j})|\right)\left\|\frac{\partial f}{\partial \lambda_{j}}\right\|_{L^{\infty}(\R^{2})}\\
\label{D-Thm1-R7}&=(\tau_{1,\psi}+\tau_{2,\psi}+\tau_{3,\psi}+\tau_{4,\psi})\left(|\Re(V_{j})|+|\Im(V_{j})|\right)\left\|\frac{\partial f}{\partial \lambda_{j}}\right\|_{L^{\infty}(\R^{2})}.
\end{align}
Hence
\begin{eqnarray}
\nonumber\left|\tau_{\psi}\left(V_{j}\frac{\partial f}{\partial\lambda_{j}}(L_{1},L_{2})\right)\right|&\stackrel{\eqref{D-Thm1-R6}}{=}&\left|\tau_{\psi}\left(P_{\hil}\,\frac{\partial f}{\partial \lambda_{j}}(X_{1},X_{2})\,V_{j}\right)\right|\\
\nonumber&\stackrel{\scriptsize\mbox{\cite[Prop. 2.3]{Sk15}}}{=}&\left|\tau_{\mathcal{I}}\left(\frac{\partial f}{\partial \lambda_{j}}(X_{1},X_{2})V_{j}\right)\right|\\
\label{D-Thm1-R8}&\stackrel{\eqref{D-Thm1-R7}}{\le}&\sum_{k=1}^{4}\tau_{k,\psi}\left(|\Re(V_{j})|+|\Im(V_{j})|\right)\left\|\frac{\partial f}{\partial \lambda_{j}}\right\|_{L^{\infty}(\R^{2})}.
\end{eqnarray}
Finally, from \eqref{D-Thm1-R5} and the estimate \eqref{D-Thm1-R8}, by following the same steps as in the proof of Theorem \ref{Kr-Thm2}, we conclude the proof.
\end{proof}

\begin{rmrk}\label{D-Rmrk1}
The proof above indicates that the assumptions of Theorem \ref{Kr-Thm3} can be relaxed. If one assumes $f\in\mathfrak{R}(\Rn)$ (cf. \eqref{rational}) instead of $f\in\mathcal{W}_{2}(\Rn)$, then the conclusion \eqref{Kr-Thm3-R3} holds for $\Hop\in\comn$ without requiring $\Hop(1)$ to be commutative. The original function class $\mathcal{W}_{2}(\Rn)$ was not regular enough to guarantee that $f(\Hop(1))$ is well-defined, unless one assumes $\Hop(1)$ is commutative.
\end{rmrk}

\subsection{Koplienko trace formula for maximal dissipative operators}
We now establish the Koplienko trace formula for a pair of maximal dissipative operators.

\begin{thm}\label{D-Thm2}
Let $\psi:(0,\infty)\to(0,\infty)$ be a concave function such that $\lim_{t\to\infty} \psi(t)=\infty$ and the condition \eqref{psi-growth-condition} is satisfied  for some $0<\varepsilon<1/3$. Let $\mathcal{M}_\psi$ be the corresponding Lorentz ideal and let $\tau_{\psi}$ be a bounded singular trace on $\mathcal{M}_\psi$. Let $V_{1},V_{2}\in\mathcal{M}_\psi^{1/2}$, and let $L_{1},L_{2}$ be two maximal dissipative operators in $\hil$ such that $L_{1}+V_{1}$ and $L_{2}+V_{2}$ are also maximal dissipative. If $(L_{1},L_{2})\in\rcomn$, then there exist finite measures $\nu_{11},\nu_{12},\nu_{22}$ on $\R^{2}$  such that 
\begin{align}
\nonumber&\tau_{\psi}\left\{f(L_{1}+V_{1},L_{2}+V_{2})-f(L_{1},L_{2})-\frac{d}{ds}\bigg|_{s=0}^{}f(L_{1}(s),L_{2}(s))\right\}\\
\label{D-Thm2-R1}=&2\int_{\R^{2}}\frac{\partial^{2} f}{\partial \lambda_{1}\,\partial \lambda_{2}}(\lambda_{1},\lambda_{2})\,d\nu_{12}(\lambda_{1},\lambda_{2})+\sum_{j=1}^{2}\int_{\R^{2}}\frac{\partial^{2} f}{\partial\lambda_{j}^{2}} (\lambda_{1},\lambda_{2})\,d\nu_{jj}(\lambda_{1},\lambda_{2}),
\end{align}
for every $f\in\mathfrak{R}(\R^{2})^-$.
\end{thm}

\begin{proof}
It is enough to consider $f\in\mathfrak{R}^{-}$ (cf. \eqref{Rational-without-span-}). Let
\[f(\lambda_{1},\lambda_{2})=(z-\lambda_{1})^{-k_{1}}(z-\lambda_{2})^{-k_{2}},\ \ \mbox{$\Im(z)<0$ and $k_{1},k_{2}\in\N$}.\] 
Similarly to the proof of Theorem \ref{Kp-Thm3}, we first establish the following identity:
\begin{align*}
\tau_\psi\left\{f(L_{1}(1),L_{2}(1))-f(L_{1},L_{2})-\frac{d}{ds}\bigg|_{s=0}^{}f(L_{1}(s),L_{2}(s))\right\}=\frac{1}{2}\tau_\psi\left(\frac{d^{2}}{ds^{2}}\bigg|_{s=0}^{}f(L_{1}(s),L_{2}(s))\right).
\end{align*}
From Lemma \ref{Kp-Lem1} (see also Remark \ref{Kp-Rmrk2}) this further implies
\begin{align}
\nonumber&\tau_\psi\left\{f(L_{1}(1),L_{2}(1))-f(L_{1},L_{2})-\frac{d}{ds}\bigg|_{s=0}^{}f(L_{1}(s),L_{2}(s))\right\}\\
\label{D-Thm2-R2}=&\tau_{\psi}\left(D_{L_{1},L_{2}}^{f}(0)\right)dt+\frac{1}{2}\sum_{j=1}^{2}\tau_{\psi}\left(D_{L_{j},L_{j}}^{f}(0)\right),
\end{align}
where $D^{f}_{L_{1},L_{1}}(0),D^{f}_{L_{1},L_{2}}(0)$ and $D^{f}_{L_{2},L_{2}}(0)$ are given by \eqref{Kp-Lem1-R4} and \eqref{Kp-Lem1-R5}. Moreover, from \eqref{Kp-Lem1-R4} and the cyclity of the trace $\tau_{\psi}$, we derive
\begin{align}
\nonumber&\tau_{\psi}\left(D^{f}_{L_{1},L_{2}}(0)\right)\\
\label{D-Thm2-R3}=&\sum_{\substack {1\le p_{0},p_{1}\le k_{1} \\ p_{0}+p_{1}=k_{1}+1}}~\sum_{\substack {1\le q_{0},q_{1}\le k_{2} \\ q_{0}+q_{1}=k_{2}+1}}\tau_{\psi}\left[(zI-L_{1})^{-p_{0}}(zI-L_{2})^{-q_{1}}V_{1}(zI-L_{1})^{-p_{1}}(zI-L_{2})^{-q_{0}}V_{2}\right].
\end{align}

Let $\tau_{\psi}=\tau_{1,\psi}-\tau_{2,\psi}+i\tau_{3,\psi}-i\tau_{4,\psi}$ be the Jordan decomposition of $\tau_{\psi}:\mathcal{M}_{\psi}\to\C$, where $\tau_{i,\psi}$ for $1\le i\le 4$ are positive bounded traces. Again, according to \cite[Proposition 2.3]{Sk15}, there exists an ideal $\mathcal{I}\subset\mathcal{B}(\mathcal{K})$ and a bounded trace $\tau_{\mathcal{I}}$ on $\mathcal{I}$ such that
\[\tau_{\mathcal{I}}=\tau_{1,\mathcal{I}}-\tau_{2,\mathcal{I}}+i\tau_{3,\mathcal{I}}-i\tau_{4,\mathcal{I}}\]
(as described in the proof of Theorem \ref{D-Thm1}) and the restrictions of $\tau_{i,\mathcal{I}}$ to $\mathcal{M}_{\psi}$ coincide with $\tau_{i,\psi}$. 

Let $(X_{1},X_{2})\in\text{Com}_{2}$ be the resolvent self-adjoint dilation of the tuple $(L_{1},L_{2})$ on a Hilbert space $\mathcal{K}\supset\hil$ such that \eqref{Dilation} holds with a projection $P_{\hil}$. Therfore \eqref{D-Thm2-R3} yields
\begin{align}
\nonumber&\tau_{\psi}\left(D_{L_{1},L_{2}}^{f}(0)\right)\\
\label{D-Thm2-R5}=&\sum_{\substack{1\le p_{0},p_{1}\le k_{1}\\p_{0}+p_{1}=k_{1}+1}}\sum_{\substack {1\le q_{0},q_{1}\le k_{2}\\q_{0}+q_{1}=k_{2}+1}}\tau_{\mathcal{I}}\big\{P_{\hil}(zI-X_{1})^{-p_{0}}(zI-X_{2})^{-q_{1}}V_{1}\,P_{\hil}(zI-X_{1})^{-p_{1}}(zI-X_{2})^{-q_{0}}\,V_{2}\big\}.
\end{align}   

Observe that $V_{1}P_{\hil},V_{2}P_{\hil}\in\mathcal{B}(\mathcal{K})\cap\mathcal{I}$. Therefore as in the proof of Theorem \ref{Kp-Thm1} ({\it Case} 1), we obtain from \eqref{D-Thm2-R5} that
\begin{align}
\nonumber\left|\tau_{\psi}\left(D_{L_{1},L_{2}}^{f}(0)\right)\right|&\le\sum_{k=1}^{4}\left(\tau_{k,\mathcal{I}}(|V_{1}P_{\hil}|^{2})\right)^{1/2}\,\left(\tau_{k,\mathcal{I}}(|V_{2}P_{\hil}|^{2})\right)^{1/2}\,\,\left\|\frac{\partial^{2} f}{\partial\lambda_{1}\,\partial \lambda_{2}}\right\|_{L^{\infty}(\R^{2})}\\
\label{D-Thm2-R4}&\le\sum_{k=1}^{4}\left(\tau_{k,\psi}(|V_{1}|^{2})\right)^{1/2}\,\left(\tau_{k,\psi}(|V_{2}|^{2})\right)^{1/2}\,\,\left\|\frac{\partial^{2} f}{\partial\lambda_{1}\,\partial \lambda_{2}}\right\|_{L^{\infty}(\R^{2})}.
\end{align}
Following the same method as above, we obtain estimates for $\left|\tau_{\psi}\left(D_{L_{j},L_{j}}^{f}(0)\right)\right|$ for $j=1,2$. The proof is completed by combining equation \eqref{D-Thm2-R2} with the estimate in \eqref{D-Thm2-R4}, following the concluding argument of Theorem \ref{Kp-Thm2}.
\end{proof}

\begin{rmrk}\label{D-Rmrk2}
Let $\il$ be a symmetrically normed ideal endowed with a bounded trace $\tau_{\il}:\il\to\C$. Let $V_{1},V_{2}\in\mathcal{B}(\hil)$, and let $L_{1},L_{2}$ be two maximal dissipative operators in $\hil$ such that $L_{1}+V_{1}$ and $L_{2}+V_{2}$ are also maximal dissipative. Assume $(L_{1}+tV_{1},L_{2}+tV_{2})\in\rcomn$ for all $t\in[0,1]$.
 
Using methods analogous to Theorems \ref{Kr-Thm2} and \ref{Kp-Thm2}, along with the proofs of Theorems \ref{D-Thm1} and \ref{D-Thm2}, one can establish the following extensions:
\begin{enumerate}[{\normalfont(i)}]
\item Under Hypothesis \ref{Assumption}(a), the identity \eqref{D-Thm1-R2} holds for perturbations in $\il$, with $\tau_{\psi}$ replaced by $\tau_{\il}$.
\item Under the Hypothesis \ref{Assumption}, the identity \eqref{D-Thm2-R1} holds for perturbations in the root ideal $\ideal$, with $\tau_{\psi}$ replaced by $\tau_{\il}$.
\end{enumerate} 
\end{rmrk}

We conclude this section with the following example.

\begin{xmpl}\label{hardy}
Let $\hil=H^{2}(\cir)$, the Hardy space over the unit circle $\cir$ (see \cite[Section 4.2 and Theorem 4.3]{MashreghiBook}). For $z\in\cir\backslash\{1\}$, consider the pair of shift operators $(M_{z}, M_{z^{2}})$ over $\hil$. Define maximal dissipative operators
$$L_{k}=i(I_{\hil}+M_{z^{k}})(I_{\hil}-M_{z^{k}})^{-1},\ \ \ k=1,2.$$
By \cite[Theorem 3.5]{AlPe22}, $L_{1}$ and $L_{2}$ commute in the resolvent sense.

For $z\in\C$ with $\Im(z)<0$ and $k_{1},k_{2}\in\N$, standard calculations yield (see, e.g., \cite[p. 1115]{DySk14}):
\begin{align}\label{D-Exm-R1}
(zI_{\hil}-L_{1})^{-k_{1}}=\sum_{n=0}^{\infty}w_{n}(k_{1},z)M_{z}^{n},\ \ (zI_{\hil}-L_{2})^{-k_{2}}=\sum_{n=0}^{\infty}w_{n}(k_{2},z)M_{z^{2}}^{n}
\end{align}
where the coefficients $w_{n}(k_{i},z)\in\C$ are such that $\sum_{n=0}^{\infty}\left|w_{n}(k_{i},z)\right|<\infty$ for $i=1,2$.

Now, let $(U_{1},U_{2})$ be a commuting unitary dilation of $(M_{z},M_{z^{2}})$ on a Hilbert space $\mathcal{K}\supset\hil$ satisfying
\begin{align}\label{D-Exm-R2}
M_{z}^{j}M_{z^{2}}^{k}=P_{\hil}U_{1}^{j}U_{2}^{k}\big|_{\hil},\ \ j,k\ge0,
\end{align}
where $P_{\hil}$ denotes the orthogonal projection from $\mathcal{K}$ onto $\hil$. Without loss of generality we may assume that $1\notin\sigma_{p}(U_{1})$ and $1\notin\sigma_{p}(U_{2})$ (see \cite[p. 1047]{AlPe22}).

Define the (possibly unbounded) self-adjoint operators
\[X_{k}=i(I_{\mathcal{K}}+U_{k})(I_{\mathcal{K}}-U_{k})^{-1},\ k=1,2.\]
Since $U_{1}$ and $U_{2}$ commute, it follows analogously that $X_{1}$ and $X_{2}$ commute in the resolvent sense. By \cite[Theorem VIII.13]{BarryBook}, this further implies that $(X_{1},X_{2})\in\text{Com}_{2}$. 
Moreover, we have
\begin{align}\label{D-Exm-R3}
(zI_{\mathcal{K}}-X_{1})^{-k_{1}}=\sum_{n=0}^{\infty}w_{n}(k_{1},z)U_{1}^{n},\ \ (zI_{\mathcal{K}}-X_{2})^{-k_{2}}=\sum_{n=0}^{\infty}w_{n}(k_{2},z)U_{2}^{n}.
\end{align}
Finally, from \eqref{D-Exm-R1}, \eqref{D-Exm-R2} and \eqref{D-Exm-R3} we see that $(L_{1},L_{2})\in\rcomn$ admits a commuting resolvent self-adjoint dilation $(X_{1},X_{2})$ satisfying \eqref{Dilation}.
\end{xmpl}

\section{Weaker trace formulas}\label{Sec6}
In Sections \ref{Sec3} and \ref{Sec4}, we established the  Krein trace formula (Theorem \ref{Kr-Thm2}) and the Koplienko trace formula (Theorem \ref{Kp-Thm2}) for perturbations in the symmetrically normal ideal $\il$ and its root ideal $\ideal$, respectively. The main tool in these proofs is the joint spectral theorem for the $n$-tuple $\Hop(t)$, which in turn requires $\Hop(t)$ to be commutative for all $t\in[0,1]$.

Recall Notation \ref{Notation}. In the present section, we establish weaker variants of these trace formulas (see \eqref{P-R5} and \eqref{P-R7}), now assuming only that the tuples $\Hop,\Hop(1)\in\comn$. The proof makes full use of the theory of multiple operator integrals discussed in Section \ref{Sec2} and does not depend on the joint spectral theorem for $\Hop(t)$ for $t\in[0,1]$.

To proceed, we need following two elementary lemmas.

\begin{lma}\label{P-Lem1}
Let $n\in\N$ and $f\in C_c^{n}((a,b)^n)$. Then for $1\le i_{1}<\cdots<i_{l}\le n$, we have
\begin{align}
\nonumber&\left\|\frac{\partial^{l}f}{\partial \lambda_{i_{1}}\cdots\partial \lambda_{i_{l}}}\right\|_{\infty}\le (b-a)^{n-l}\left\|\frac{\partial^{n}f}{\partial \lambda_{1}\cdots\partial \lambda_{n}}\right\|_{\infty},\quad l=0,\ldots,n.
\end{align}
\end{lma}

\begin{lma}\label{P-Lem2}
Let $n\in\N$. Then $C_c^{n+1}((a,b)^n)\subseteq\Wn$.
\end{lma}

\begin{proof}
A routine calculation shows that for $f\in C_c^{n}((a,b)^n)$,
\begin{align}
\label{P-R1}&\|\widehat{f}\|_{L^{1}(\Rn)}\le(\sqrt{2})^{n}\bigg[\|f\|_{L^{2}(\Rn)}+\sum_{l=1}^{n}\bigg\{\sum_{1\le i_{1}<\cdots<i_{l}\le n}\left\|\frac{\partial^{l}f}{\partial \lambda_{i_{1}}\cdots\partial \lambda_{i_{l}}}\right\|_{L^{2}(\Rn)}\bigg\}\bigg].
\end{align}
The proof of \eqref{P-R1} for $n=1$ can be found in \cite[Lemma 7]{PoSu09}.

Consequently, \eqref{P-R1} implies that $C_c^{n+1}((a,b)^n) \subseteq\Wn$.
\end{proof}

\begin{rmrk}
Using \eqref{P-R1} one can similarly show that $C_{c}^{n+2}((a,b)^{n})\subseteq\mathcal{W}_{2}(\Rn)$.
\end{rmrk}

Recall the notation from Notation \ref{Notation}. Let $\il$ be a symmetrically normed ideal of $\bh$, equipped with a $\|\cdot\|_{\il}$-bounded trace $\Tril:\il\to\C$. Let $H_{j}\in\bh$ and $V_{j}\in\il$ be self-adjoint operators for $j=1,\ldots,n$ such that $\Hop,\Hop(1)\in\comn$. Let $a,b\in\R$ satisfy
$$\conv\{\sigma(H_{1}+V_{1})\cup\sigma(H_{1})\}\times\cdots\times\conv\{\sigma(H_{n}+V_{n})\cup\sigma(H_{n})\}\subset(a,b)^{n}.$$ 
Let $f\in\Ccrr$. By Lemmas \ref{Kr-Lem1} and \ref{P-Lem2}, the operators $f(\Hop)$ and $f(\Hop(1))$ are well defined. Lemma \ref{Kr-Lem1} together with Theorem \ref{MOI-Thm2} ensures that
\begin{align*}
f(\Hop(1))-f(\Hop)\in\il.
\end{align*}
Moreover
\begin{align}
\nonumber f(\Hop(1))-f(\Hop)&=\sum_{j=1}^{n}T_{f_{j}^{[1]}}^{H_{1},\ldots,H_{j-1},H_{j}+V_{j},H_{j},H_{j+1}+V_{j+1},\ldots,H_{n}+V_{n}}(I,\ldots,I,\underbrace{V_{j}}_{j},I,\ldots,I)\\
\nonumber&=:\sum_{j=1}^{n}\Gamma_{j}(f,\Hop,\V).
\end{align}
Proceeding further with Corollary \ref{MOI-Crl}, we obtain
\begin{align}
\label{P-R3}\left|\Tril\left\{\Gamma_{j}(f,\Hop,\V)\right\}\right|\le\|\Tril\|_{\il^{*}}\frac{1}{(2\pi)^{n/2}}\|V_{j}\|_{\il}\left\|\widehat{\frac{\partial f}{\partial\lambda_{j}}}\right\|_{L^{1}(\Rn)},\ \ \mbox{for $1\le j\le n$}.
\end{align}
Applying \eqref{P-R1} and Lemma \ref{P-Lem1} in \eqref{P-R3} we further derive
\begin{align}
\label{P-R4}&\left|\Tril\left\{\Gamma_{j}(f,\Hop,\V)\right\}\right|\le\text{C}_{a,b,n}\|\Tril\|_{\il^{*}}\|V_{j}\|_{\il}\left\|\frac{\partial^{n+1}f}{\partial \lambda_{1}\cdots\partial \lambda_{j}^{2}\cdots\partial \lambda_{n}}\right\|_{\infty},
\end{align}
where $$\text{C}_{a,b,n}=\frac{1}{\pi^{n/2}}(b-a)^{n/2}(b-a+1)^{n}.$$
Finally, applying \eqref{P-R4}, the Hahn-Banach theorem, and the Riesz representation theorem, we deduce the existence of measures $\mu_{1},\ldots,\mu_{n}$ such that
\begin{align}
\label{P-R5}&\Tril\left\{f(\Hop(1))-f(\Hop)\right\}=\sum_{j=1}^{n}\int_{[a, b]^n}\frac{\partial^{n+1}f}{\partial \lambda_{1}\cdots\partial \lambda_{j}^{2}\cdots\partial \lambda_{n}}(\lambda_{1},\ldots,\lambda_{n})\,d\mu_{j}(\lambda_{1},\ldots,\lambda_{n}).
\end{align}

Now we proceed to prove a weaker variant of the Koplienko trace formula. For this we work under Hypothesis \ref{Assumption}.

Let $H_{j}\in\bh$ and $V_{j}\in\ideal$ be self-adjoint operators for $j=1,\ldots,n$, and assume $\Hop,\Hop(1)\in\comn$. Let $f\in C_{c}^{n+2}((a,b)^{n})$. Since Lemma \ref{Kr-Lem2} requires $\Hop(t)\in\comn$ for all $t\in[0,1]$ in order to define $\dds\bigg|_{s=0}f(\Hop(s))$. Therefore, we consider the following: 
\[(\Delta f)(\Hop,\V):=f(\Hop(1))-f(\Hop)-\dds\bigg|_{s=0}T_{f}^{H_{1}(s),\ldots,H_{n}(s)}(I,\ldots,I).\]

Using Lemma \ref{Kr-Lem2} and Theorem \ref{MOI-Thm2}, we obtain
\begin{align}
\nonumber&(\Delta f)(\Hop,\V)\\
\nonumber&=\sum_{1\le i<j\le n}T_{f_{i,j}^{[2]}}^{X_{1}^{i-1}(0),H_{i}+V_{i},H_{i},X_{i+1}^{j-1}(1),H_{j}+V_{j},H_{j},X_{j+1}^{n}(0)}(I,\ldots,I,\underbrace{V_{i}}_{i},I,\ldots,I,\underbrace{V_{j}}_{j+1},I,\ldots,I)\\
\label{P-R6}&\quad+\sum_{1\le j\le n}T_{f_{j}^{[2]}}^{X_{1}^{j-1}(0),H_{j}+V_{j},H_{j},H_{j},X_{j+1}^{n}(0)}(I,\ldots,I,\underbrace{V_{j}}_{j},V_{j},I,\ldots,I),
\end{align}
where we have used the following notation for $l,m\in\N$:
\[X_{l}^{m}(t)=\begin{cases}
H_{l}+tV_{l},\ldots,H_{m}+tV_{m},&\mbox{if }1\le l\le m\le n,\\
\varnothing,&\mbox{elsewhere}.
\end{cases}\]
Then, in a similar spirit to the proof of \eqref{P-R5}, and utilizing \eqref{MOI-Crl-R2} and \eqref{MOI-Crl-R3} in \eqref{P-R6}, we can show the existence of measures $\nu_{ij},1\le i\le j\le n$ such that
\begin{align}
\nonumber\Tril\left\{(\Delta f)(\Hop,\V)\right\}&=\sum_{1\le i<j\le n}\int_{[a, b]^n}\frac{\partial^{n+2}f}{\partial \lambda_{1}\cdots\partial \lambda_{i}^{2}\cdots\partial \lambda_{j}^{2}\cdots\partial \lambda_{n}}(\lambda_{1},\ldots,\lambda_{n})\,d\nu_{ij}(\lambda_{1},\ldots,\lambda_{n})\\
\label{P-R7}&\quad+\sum_{1\le j\le n}\int_{[a, b]^n}\frac{\partial^{n+2}f}{\partial \lambda_{1}\cdots\partial \lambda_{j}^{3}\cdots\partial \lambda_{n}}(\lambda_{1},\ldots,\lambda_{n})\,d\nu_{jj}(\lambda_{1},\ldots,\lambda_{n}).
\end{align}

\vspace*{0.3cm}	
\noindent{\it Acknowledgements:}
A. Chattopadhyay is supported by the Core Research Grant (CRG), File No: CRG/2023/004826, by the Science and Engineering Research Board (SERB), Department of Science \& Technology (DST), Government of India. S. Giri acknowledges the support provided by the Prime Minister's Research Fellowship (PMRF), PMRF-ID: 1902164, Government of India. C. Pradhan acknowledges support from the NBHM post-doctoral fellowship, Government of India. A. Usachev was partially supported by the Theoretical Physics and Mathematics Advancement Foundation \enquote{BASIS}. The authors thank the anonymous referees for a detailed reading of the text and a number of comments which improved the exposition.

\vspace{.1in}

\noindent{\bf Data Availability Statement:} Data sharing is not applicable to this article as no data sets were generated or analyzed during the current study.

\section*{Declarations} 	

\noindent{\textbf{Conflicts of Interest}:} The authors declare that they have no conflict of interest.

\end{document}